\documentclass[oneside,reqno,a4paper,11pt]{amsart}
\usepackage{amsfonts}
\usepackage{amsmath,amsfonts,amssymb}
\usepackage{CJK,bbm,color,soul,supertabular,longtable,verbatim}
\usepackage[warning, easyold,math]{easyeqn}
\usepackage{graphicx}
\usepackage{color}
\usepackage{latexsym, amsfonts,mathrsfs, amsmath, amssymb, amscd, epsfig}
\usepackage{mathrsfs}
\usepackage{ esint }
\usepackage{amsthm}
\usepackage{etex}
\usepackage{epstopdf}

\newtheorem {theorem}{Theorem}[section]
\newtheorem {lemma}[theorem]{{\bf Lemma}}

\newtheorem {corollary}[theorem]{{\bf Corollary}}
\newtheorem {prop}[theorem]{{\bf Proposition}}
\theoremstyle{remark}
\newtheorem {remark}{{\bf Remark}}[section]

\theoremstyle{problem}

\theoremstyle{definition}
\newtheorem {definition}{{\bf Definition}}[section]
\theoremstyle{plain} \numberwithin {equation}{section}

\begin{document}
\vspace{1cm}

\title[Oblique Injection of incompressible ideal fluid]{Oblique Injection of incompressible ideal fluid from a slot into a free stream$^*$}
\author[Jianfeng Cheng,\\Lili Du]{Jianfeng Cheng$^{1,2}$,\ \ Lili Du$^{1,\dag}$}

\thanks{$^*$Cheng is supported in part by NSFC grant 12001387 and the Fundamental Research Funds for the Central Universities No. YJ202046, Du is supported by NSFC grant 11971331.}
\thanks{ E-Mail: jianfengcheng@126.com (J. Cheng), E-mail: dulili@scu.edu.cn (L. Du).}
\thanks{$^\dag$Corresponding author.}
 \maketitle
\begin{center}

 $^1$ Department of Mathematics, Sichuan University,

          Chengdu 610064, P. R. China.

$^2$ The Institute of Mathematical Sciences,

The Chinese University of Hong Kong,

Shatin, N.T., Hong Kong.

\end{center}

\begin{abstract} This paper deals with a two-phase fluid free
boundary problem in a slot-film cooling. We give two well-posedness
results on the existence and uniqueness of the incompressible
inviscid two-phase fluid with a jump relation on free interface. The
problem formulates the oblique injection of an incompressible ideal
fluid from a slot into a free stream. From the mathematical point of
view, this work is motivated by the pioneer work \cite{FA3} by A.
Friedman, in which some well-posedness results are obtained in some
special case.
Furthermore, A. Friedman proposed an open problem in \cite{FA4} on
the existence and uniqueness of the injection flow problem for more
general case. The main results in
this paper solve the open problem and establish the well-posedness results on the physical problem. 

\end{abstract}

\

\begin{center}
\begin{minipage}{5.5in}

\

\end{minipage}
\end{center}

\

\everymath{\displaystyle}
\newcommand {\eqdef }{\ensuremath {\stackrel {\mathrm {\Delta}}{=}}}


\def\Xint #1{\mathchoice
{\XXint \displaystyle \textstyle {#1}} %
{\XXint \textstyle \scriptstyle {#1}} %
{\XXint \scriptstyle \scriptscriptstyle {#1}} %
{\XXint \scriptscriptstyle \scriptscriptstyle {#1}} %
\!\int}
\def\XXint #1#2#3{{\setbox 0=\hbox {$#1{#2#3}{\int }$}
\vcenter {\hbox {$#2#3$}}\kern -.5\wd 0}}
\def\ddashint {\Xint =}
\def\dashint {\Xint -}
\def\clockint {\Xint \circlearrowright } 
\def\counterint {\Xint \rotcirclearrowleft } 
\def\rotcirclearrowleft {\mathpalette {\RotLSymbol { -30}}\circlearrowleft }
\def\RotLSymbol #1#2#3{\rotatebox [ origin =c ]{#1}{$#2#3$}}

\def\aint{\dashint}

\def\arraystretch{2}
\def\eps{\varepsilon}

\def\s#1{\mathbb{#1}} 
\def\t#1{\tilde{#1}} 
\def\b#1{\overline{#1}}
\def\N{\mathcal{N}} 
\def\M{\mathcal{M}} 
\def\R{{\mathbb{R}}}
\def\B{{\mathcal{B}}}
\def\BB{\mathfrak{B}}
\def\F{{\mathcal{F}}}
\def\G{{\mathcal{G}}}
\def\ba{\begin{array}}
\def\ea{\end{array}}
\def\be{\begin{equation}}
\def\ee{\end{equation}}

\def\bes{\begin{mysubequations}}
\def\ees{\end{mysubequations}}

\def\cz#1{\|#1\|_{C^{0,\alpha}}}
\def\ca#1{\|#1\|_{C^{1,\alpha}}}
\def\cb#1{\|#1\|_{C^{2,\alpha}}}
\def\psir{\left|\frac{\nabla\psi}{r}\right|^2}
\def\lb#1{\|#1\|_{L^2}}
\def\ha#1{\|#1\|_{H^1}}
\def\hb#1{\|#1\|_{H^2}}
\def\th{\theta}
\def\Th{\Theta}
\def\cin{\subset\subset}
\def\Ld{\Lambda}
\def\ld{\lambda}
\def\ol{{\Omega_L}}
\def\sla{{S_L^-}}
\def\slb{{S_L^+}}
\def\e{\varepsilon}
\def\C{\mathbf{C}} 
\def\cl#1{\overline{#1}}
\def\ra{\rightarrow}
\def\xra{\xrightarrow}
\def\g{\nabla}
\def\a{\alpha}
\def\b{\beta}
\def\s{\sigma}
\def\d{\delta}
\def\th{\theta}
\def\fai{\varphi}
\def\O{\Omega}
\def\ol{{\Omega_L}}
\def\psirk{\left|\frac{\nabla\psi}{r+k}\right|^2}
\def\tO{\tilde{\Omega}}
\def\tu{\tilde{u}}
\def\tv{\tilde{v}}
\def\trho{\tilde{\rho}}
\def\W{\mathcal{W}}
\def\f{\frac}
\def\p{\partial}
\def\m{\omega}
\def\B{\mathcal{B}}
\def\H{\Theta}
\def\msS{\mathscr{S}}
\def\bq{\mathbf{q}}
\def\msE{\mathcal{E}}
\def\mfa{\mathfrak{a}}
\def\mfb{\mathfrak{b}}
\def\mfc{\mathfrak{c}}
\def\mfd{\mathfrak{d}}
\def\Div{\text{div}}
\def\Rot{\text{rot}}
\def\Curl{\text{curl}}
\def\mcL{\mathcal{L}}
\def\mcR{\mathcal{R}}
\def\f{\frac}
\def\p{\partial}
\def\o{\omega}
\def\h{_2^{\frac{1}{2}}}
\def\hh{_2^2}
\def\hhh{_2^{\frac{2}{3}}}
\def\k{_2^{\frac{3}{2}}}
\def\ii{\int_{0}^{t}\int}
\def\xiao{\leq}

\section{Introduction and main results}

\subsection{Introduction}

This paper is concerned with a two-phase free boundary problem
produced when a secondary fluid (or injected or coolant) is injected
obliquely at an angle from a slot into a cross flow fluid (see
Figure \ref{f1}). One important physical situation in which this
problem arises in fuel injectors, smokestacks, the cooling of
gas-turbine blades, and dilution holes in gas turbine combustors.
Please see the review of this physical problem \cite{MO}. Many
numerical simulations on this problem were investigated in
\cite{CHO,MM,SM} and the references therein.

\begin{figure}[!h]
\includegraphics[width=100mm]{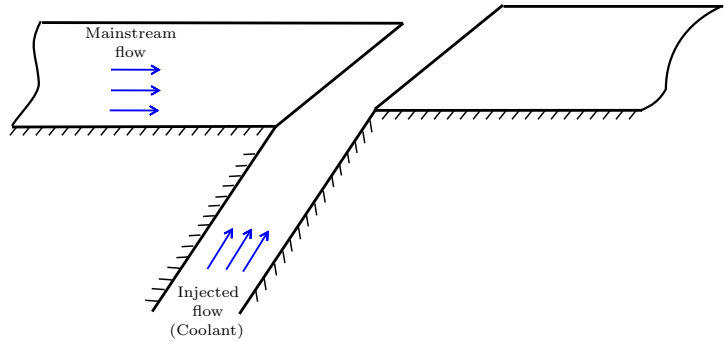}
\caption{Injection of flow from a slot into a cross flow}\label{f1}
\end{figure}

\begin{figure}[!h]
\includegraphics[width=100mm]{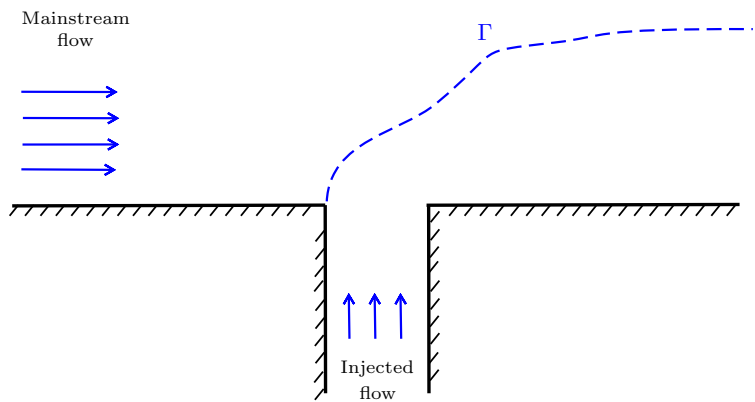}
\caption{Horizontal blade surfaces}\label{f2}
\end{figure}

Mathematically, the motivation to investigate this free boundary
problem follows from the work \cite{FA3} by A. Friedman. He first
considered the two-dimensional model and the simple situation of
horizontal blade surface and the secondary fluid injected
perpendicularly into a free stream in two dimensions (as in Figure
\ref{f2}). Here, for simplicity, we neglect the separation at the
trailing edge $B$ of the slot, such separation can be minimized in
practice by slightly around the trailing edge. Also, we have assumed
that the interface between the main stream flow and the secondary
flow separates at the leading edge $A$, since the viscosity effects
are ignored. Some existence and uniqueness of the solution to the
two-phase fluid were established for simple special case in
\cite{FA3}. And furthermore, A. Friedman proposed an open problem in
Page 69 in his survey \cite{FA4}, that
$$\ba{rl}\text{{\it Problem (1)}.} &\text{{\it Extend the results of Theorem 9.1, 9.2 to more general
geometries,}}\\
&\text{{\it such as in Figure 9.3}}.\ea$$

\begin{figure}[!h]
\includegraphics[width=100mm]{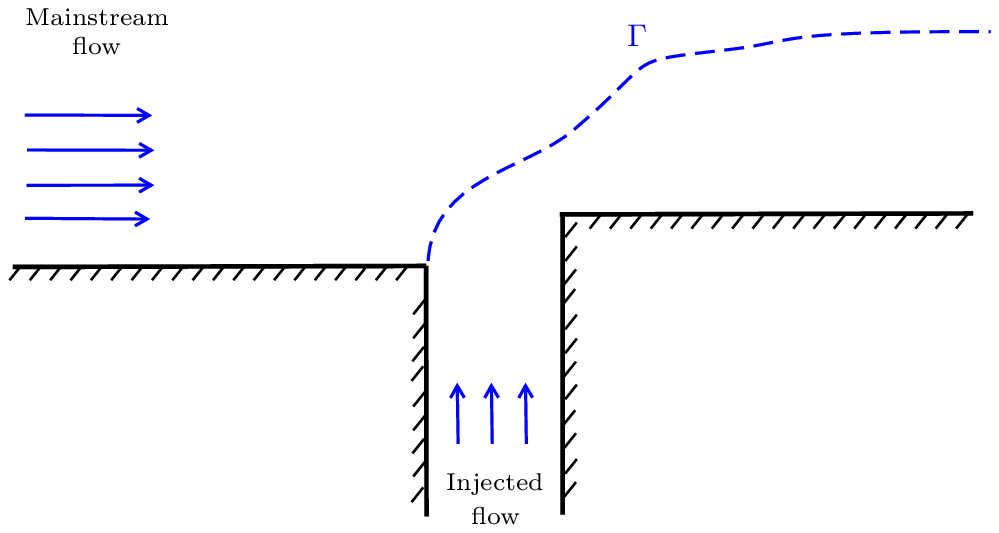}
\caption{Figure 9.3 in \cite{FA4}}\label{f3}
\end{figure}

Please see Figure \ref{f3} for the Figure 9.3 in \cite{FA4}.

The main purpose in this paper is to establish the existence and
uniqueness of the free boundary problem on an incompressible
inviscid fluid obliquely into a free stream (as in Figure \ref{f4})
and solve the open problem proposed by A. Friedman.

\begin{figure}[!h]
\includegraphics[width=100mm]{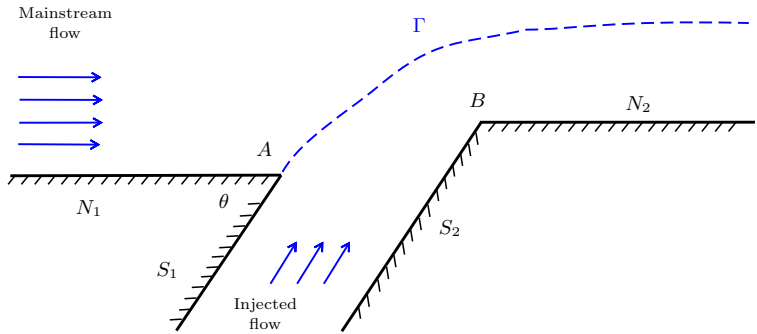}
\caption{Oblique injection of flow into a free stream}\label{f4}
\end{figure}

In general, there is a discontinuity in the magnitude of velocity
across the interface due to the Bernoulli's law. Therefore, the
standard method of conformal mapping from the complex potential
plane to the conjugate velocity plane will not be fruitful because
the interface is mapping into unknown curve in the conjugate
velocity plane. For the special case of $\th=\f\pi2$, the free
boundary problem was reduced to a nonlinear singular integral
differential equation in \cite{TL}. Along the variational arguments
introduced in \cite{ACF3,ACF4,ACF5,ACF6}, A. Friedman established
the well-posedness results for the some special case ($b=0$ and
$\th=\f\pi2$) in Figure \ref{f2}.

\subsection{Notations and the free boundary problem} Before we state
the main results in this paper, we will give the following notations
of the geometry of the blade surface.

 Denote $$N_1=\{(x,0)\mid x\leq0\},\ \ N_2=\{(x,b)\mid x\geq a\}.$$ Here,
 $a>0$, and we consider the general case that the blade surfaces are
 not horizontal, namely, $b>0$. Let
 $$S_1=\{(x,y)\mid x=y\cot\th, y\leq0\}\ \ \text{and}\ \ S_2=\{(x,y)\mid x=(y-b)\cot\th+a,
 y\leq b\},$$ where $\th\in\left(0,\f\pi2\right]$. Furthermore, we assume $a\sin\th-b\cos\th>0$, which excludes the
 possibility of the intersection of $N_1$ and $S_2$, $\th$ is the inclination and the critical
 case $\th=\f\pi2$ means the normal injection. The leading edge
 $A=(0,0)$ and the trailing edge $B=(a,b)$.

Both of the mainstream flow and the secondary flow are assumed to be
steady, incompressible, inviscid and irrotational. Denote by
$(u_+,v_+), p_+,\rho_+$ the velocity, pressure and the constant
density of the mainstream flow in $\O^+$, and $(u_-,v_-),
p_-,\rho_-$ as the velocity, pressure and the constant density of
the secondary flow in $\O^-$. They are separated by a streamline,
denoted as $\Gamma$. The pressure across the interface $\Gamma$ has
to be continuous, i.e., $p_+=p_-$ on $\Gamma$. We assume that the
mainstream flow is horizontal and possesses a uniform speed $U_0$ in
upstream, without loss of generality, $U_0=\f1{\sqrt{\rho_+}}$. The
secondary flow with mass flux $Q_0$ emerges from a slot, where the
magnitude of $Q_0$ is unrestricted  for the moment.


Denote by $\O$ the fluid field of the two-phase fluid, composed of
the following domains
$$\text{$T_1=\{x\leq 0,
y\geq0\}$, $T_2=\{y\cot\th\leq x\leq0,y\leq 0\}$,}$$ and
$$T_3=\{0\leq
x\leq\min\{(y-b)\cot\th+a,a\}\}\ \ \text{and}\ \ T_4=\{x\geq a,y\geq
b\},$$ namely, $\O=\text{int}(T_1\cup T_2\cup T_3\cup T_4)$.

Define a stream function $\psi$ of the two-phase fluid as
$$\f{\p\psi}{\p x}=-\sqrt{\rho_+}v_+\ \ \text{and}\ \ \f{\p\psi}{\p y}=\sqrt{\rho_+}u_+\ \ \text{in the main
fluid field  $\O^+$,}$$ and
$$\f{\p\psi}{\p x}=-\sqrt{\rho_-}v_-\ \ \text{and}\ \ \f{\p\psi}{\p y}=\sqrt{\rho_-}u_-\ \ \text{in the
secondary fluid field  $\O^-$}.$$

On the solid boundaries, we impose that \be\label{a0}\text{$\psi=0$
on $N_1\cup S_1$, and $\psi=-\f{Q_0}{\sqrt{\rho_-}}$ on $N_2\cup S_2$}.\ee On the
interface $\Gamma=\O\cap\{\psi=0\}$, the Bernoulli's equation gives
that
\be\label{a01}\rho_-(u_-^2+v_-^2)-\rho_+(u_+^2+v_+^2)=\text{constant,}\
\ \ \text{on}\ \ \Gamma,\ee the jump constant is denoted as $\ld$.
It is easy to see that $\ld\in(-1,+\infty)$. The two-phase fluid we
seek in this paper is the vortex sheet solution and the jump
condition \eqref{a01} is in fact the Rankine-Hugoniot jump condition
to the vortex sheet. From the mathematical point of view, to attack
the well-posedness of the problem on the injection of ideal fluid
from a slot into a free stream in 1983, A. Friedman in \cite{FA3}
(see also the Chapter 9 in \cite{FA4}) introduced the injection flow
problems in two different situations.

{\bf The injection flow problem 1.} For any given $Q_0>0$, does there
exist a unique injection flow $(\psi,\Gamma)$, such that the
mainstream flow possesses uniform speed in upstream, and the
interface $\Gamma$ connects at $A$ and extends to infinity?

{\bf The injection flow problem 2.} For any given
$\ld\in(-1,+\infty)$, does there exist a unique injection flow
$(\psi,\Gamma)$, such that the mainstream flow possesses uniform
speed in upstream, and the interface $\Gamma$ connects at $A$ and
extends to infinity?

Here, it is worth to mention that once the stream function $\psi$ is solved,
$$(u_+,v_+)=\f1{\sqrt{\rho_+}}\left(\f{\p\psi}{\p y},-\f{\p\psi}{\p x}\right)\ \ \text{in the main
fluid field  $\O^+$,}$$ and
$$(u_-,v_-)=\f1{\sqrt{\rho_-}}\left(\f{\p\psi}{\p y},-\f{\p\psi}{\p x}\right)\ \ \text{in the
secondary fluid field  $\O^-$}$$ will be solved by the stream
function.

Meanwhile, the existence and uniqueness of the injection flow
problem 1 and 2 in some geometric special situation were established
in Theorem 3.2 and Theorem 1.1 in \cite{FA3}, respectively. He
assumed that the blade surface is horizontal and the injection is
vertical, namely, $b=0$ and $\th=\f\pi2$ (see Figure \ref{f2}).
Moreover, he proposed an open problem in \cite{FA4} to extend the
results in \cite{FA3} to more general case as in Figure \ref{f3}.
This is the main motivation to investigate the oblique injection
flow problem in this paper.

Next, we will define the solution to the injection flow problem 1
and problem 2, respectively.

\begin{definition}\label{def1}
{\bf(A solution to the injection flow problem 1).}\\
For any given $Q_0>0$, a vector $(\psi,\Gamma)$ is called a solution
to the injection flow
problem 1, provided that \\
(1) $\Delta\psi=0$ in $\O\setminus \Gamma$, $\psi\in C^0(\bar\O)$
and $\g\psi\in
L^{\infty}(\O\setminus B_\e(B))$ for any $\e>0$.\\
(2) $\psi$ satisfies the Dirichlet boundary conditions \eqref{a0}. \\
 (3) The free
boundary $\Gamma: y=k(x)$ is $C^1$-smooth
 strictly increasing function in $(0,+\infty)$, and $k(x)>b$
 for any $x\geq a$. Furthermore,
\be\label{a2}k(0)=0,\ee and there exists a $h\in(b,+\infty)$, such that $$
\lim_{x\rightarrow+\infty}k(x)=h\ \ \text{and}\ \
\lim_{x\rightarrow+\infty}k'(x)=0.$$ \\
(4) $\psi$ satisfies the Rankine-Hugoniot jump condition on
$\Gamma$, namely,
\be\label{a00}\left(\f{\p\psi^-}{\p\nu}\right)^2-\left(\f{\p\psi^+}{\p\nu}\right)^2=\ld
~~\text{on}\ \Gamma,\ee where $\ld=\f1{\rho_-}\left(\f{Q_0}{h-b}\right)^2-1$,
$\psi^+=\max\{\psi,0\}$, $\psi^-=-\min\{\psi,0\}$ and
$\nu$ is the normal vector to $\Gamma$.\\
(5) $\Gamma$ is continuously differentiable at $A$ and
\be\label{a3}k'(0+0)=\left\{\ba{ll} \tan\th,\
&\text{if $\ld>0$, (see Figure \ref{f4})}\\
0 ,\
&\text{if $\ld<0$, (see Figure \ref{f5})}\\
\tan\f{\th}2,\ &\text{if\ $\ld=0$, (see Figure
\ref{f6})}.\ea\right.\ee

\begin{figure}[!h]
\includegraphics[width=90mm]{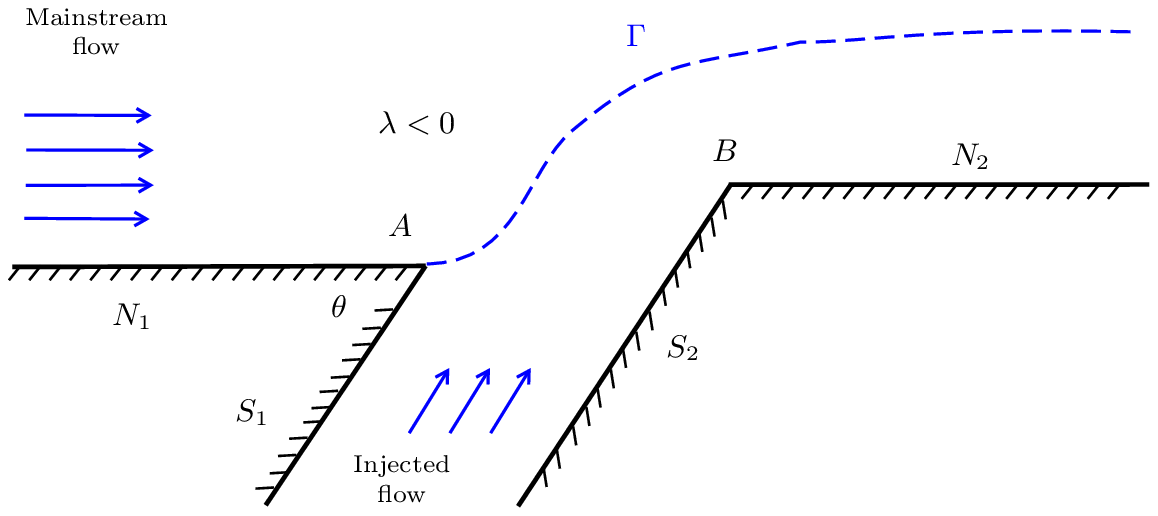}
\caption{The case $\lambda < 0$.}\label{f5}
\end{figure}

\begin{figure}[!h]
\includegraphics[width=90mm]{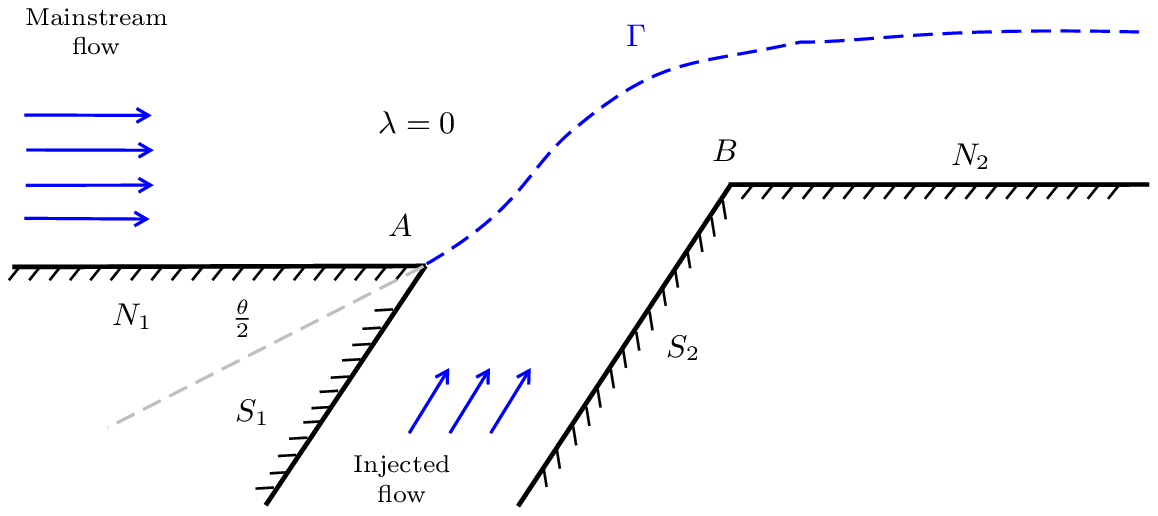}
\caption{The case $\lambda = 0$}\label{f6}
\end{figure}

(6) $\psi$ possesses the
following asymptotic behaviors in far field
$$\psi(x,y)\rightarrow\left\{\ba{ll} \f{Q_0(y-h)}{\sqrt{\rho_-}(h-b)},\
&\text{if $b<y<h$, as $x\rightarrow+\infty$},\\
y-h,\ &\text{if $h<y<C$, as $x\rightarrow+\infty$, for any
$C>0$},\ea\right.$$
$$\psi(x,y)\rightarrow y,\ \text{if $0<y<C$, as $x\rightarrow-\infty$, for any
$C>0$},$$ and
$$\left|\psi(x,y)-\f{Q_0(y\cos\th-x\sin\th)}{\sqrt{\rho_-}(a\sin\th-b\cos\th)}\right|\ra
0 \ \ \text{uniformly in any compact subset of $S$,}$$ as
$y\ra-\infty$, where $S=\{(x,y)\mid
y\cot\th<x<(y-b)\cot\th+a,-\infty<y<+\infty\}$. Furthermore,
$$\g\psi(x,y)\rightarrow (0,1),\ \text{if $x^2+y^2\rightarrow+\infty$, dist($(x,y),\Gamma$)$\rightarrow+\infty$ and with $\psi(x,y)>0$},$$ and
$$\left|\g\psi(x,y)-\f{Q_0}{\sqrt{\rho_-}(a\sin\th-b\cos\th)}\left(-\sin\th,\cos\th\right)\right|\ra
0 \ \ \text{uniformly in any compact subset of $S$,}$$ as
$y\ra-\infty$.\\ 
(7) The following estimates hold,
\be\label{a4}\text{$-h\leq\psi^+(x,y)-y\leq 0\ $ in $\
\O\cap\{y>0\}$.}\ee

\end{definition}

\begin{definition}\label{def11}
{\bf(A solution to the injection flow problem 2).} For some given
appropriate $\ld\in(-1,+\infty)$, $(\psi,\Gamma)$ is called a
solution to the injection flow problem 2, provided that the
conditions (1) - (7) in Definition \ref{def1} hold.

 \end{definition}

\begin{remark} $k(0)=0$ is nothing but the {\it continuous fit
condition} of the interface $\Gamma$, which gives that the interface
$\Gamma$ initiates at the leading edge $A$. Since the viscous
effects are ignored here, and the boundary layer is not considered,
the continuous fit condition seems to be reasonable. Moreover, the
condition \eqref{a3} is so-called {\it smooth fit condition} for
$\ld\neq0$ (please see Figure \ref{f4} and Figure \ref{f6}). 
\end{remark}

\begin{remark}
The conditions $\lim_{x\rightarrow+\infty}k(x)=h<+\infty$ and
$\lim_{x\rightarrow+\infty}k'(x)=0$ in \eqref{a2} imply that the
interface $\Gamma$ is flat and does not oscillate in downstream.
\end{remark}

\begin{remark}
To attack the injection flow problem 1, we first regard the constant
$\ld$ as an undetermined parameter, and then the parameter $\ld$
will be determined uniquely by the continuous fit condition. It
means that there exists a unique $\ld$ such that the interface
connects at the leading edge point $A$. On another hand, the
asymptotic behavior in downstream gives the relation
$h=b+\f{Q_0}{\sqrt{\rho_-(\ld+1)}}$. Once the constant $\ld$ is fixed by the
continuous fit condition, the asymptotic width $h$ can be determined
by the formula.
\end{remark}

\subsection{Main results}

For the special case $b=0$ and $\th=\f\pi2$, the existence and
uniqueness were established in \cite{FA3}, and we will give the
existence and uniqueness results on the injection flow problem in
two situations in general case as follows.
\begin{theorem}\label{th1} For any $Q_0>0$, there exist a unique $\ld>-1$ and a unique solution
$(\psi,\Gamma)$ to the injection flow problem 1. Furthermore, the
interface $\Gamma$ is analytic, $u_{\pm}>0$ in $\O^{\pm}\cup\Gamma$,
and $v_{\pm}>0$ in $\O^{\pm}\cup\Gamma$.
\end{theorem}

\begin{remark} In \cite{ACF5}, some well-posedness results on two
fluids of steady incompressible inviscid flows issuing from two
nozzles were established (see Figure \ref{f7}). However, it is
assumed that the two nozzles are symmetric with respect to $x$-axis
and the upper boundary of the nozzle I coincides with the lower
boundary of the nozzle II. Along the proof of Theorem \ref{th1} in
this paper, we can extend the existence and uniqueness of the two
fluids in \cite{ACF5} to the general case as Figure \ref{f8} (the
nozzle are asymmetric and the nozzle walls do not coincide) without
any additional difficulties.

\begin{figure}[!h]
\includegraphics[width=80mm]{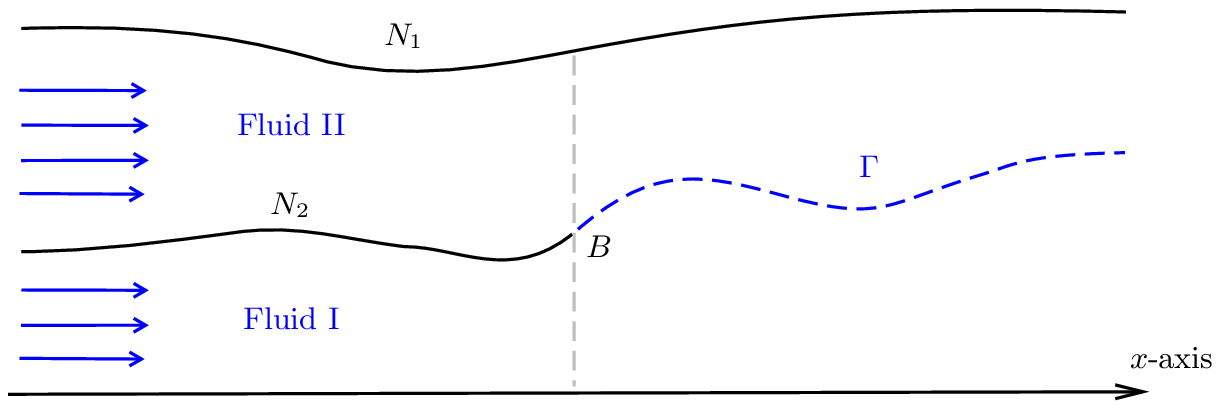}
\caption{Two-phase fluid}\label{f7}
\end{figure}

\begin{figure}[!h]
\includegraphics[width=80mm]{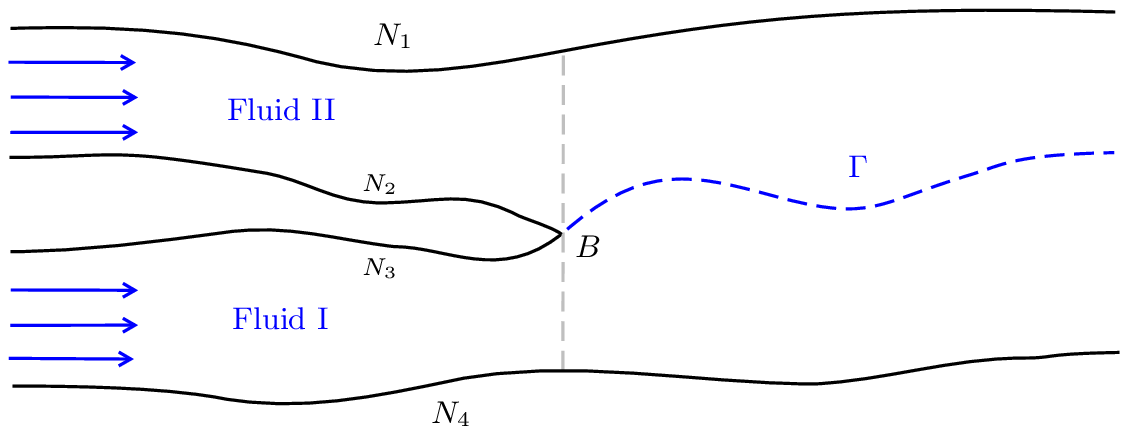}
\caption{Two-phase fluid}\label{f8}
\end{figure}

\end{remark}

\begin{remark} In the previous work \cite{FA3}, A. Friedman showed that the free boundary $\Gamma$ is
only $C^1$-smooth for $b=0$ and $\th=\f\pi2$, and then the
Rankine-Hugoniot \eqref{a00} holds in weak sense. However, we would
like to emphasize that here we showed that the free boundary is
analytic and then the Rankine-Hugoniot \eqref{a00} holds in
classical sense.

\end{remark}

On another hand, to obtain the well-posedness results on the
injection flow problem 2, we will investigate the relationship of
the constant $\ld$ and the flux of injection flow $Q_0>0$. In fact, we
show that $\ld$ is strictly monotone increasing and continuous with
respect to $Q_0>0$, denoted as $\ld=\ld(Q_0)$.

\begin{theorem}\label{th2} For any $Q_0>0$, the solution $(\psi,\Gamma,\ld(Q_0))$ established in Theorem \ref{th1} satisfies that\\
(1) $\ld(Q_0)$ is strictly monotone increasing and continuous with respect to
$Q_0>0$.\\
(2) There exists a $\underline\ld\in(-1,0)$, such that
$\ld(Q_0)\rightarrow
\underline\ld$ as $Q_0\rightarrow0$.\\
(3) There exists a $\kappa\in(0,+\infty)$, such that
$\f{\ld(Q_0)}{Q_0^2}\rightarrow \kappa$ as $Q\rightarrow+\infty$.

\end{theorem}

The second statement in Theorem \ref{th2} implies that the lower
bound is $\underline\ld$, so we can establish the well-posedness
result to the injection flow problem 2.

\begin{theorem}\label{th3} There exists a $\underline\ld>-1$ ($\underline\ld$ is given in Theorem \ref{th2}), such that
for any $\ld>\underline\ld$, there exist a unique $Q_0>0$ and a unique solution
$(\psi,\Gamma)$ to the injection flow problem 2. Furthermore,
$u_{\pm}>0$ in $\O^{\pm}\cup\Gamma$, and $v_{\pm}>0$ in
$\O^{\pm}\cup\Gamma$.
\end{theorem}

\begin{remark} Similar to the Theorem \ref{th1}, when the constant $\ld$ is imposed,
the flux of the injection flux $Q_0$ can be regarded as a parameter to
solve the injection flow problem 2. And the unique solvability of
the flux $Q_0$ can be determined by the continuous fit condition. In
particular, for $\ld=0$, the stream function is harmonic in the
whole fluid field $\O$, and the flux $Q_0>0$ is uniquely determined by
the following formula
$$(a\sin\th-b\cos\th)^{\f\pi\th}=\left(\f\pi\th-1\right)\left(\f{Q_0}{\sqrt{\rho_-}}\right)^{\f\pi\th}+\f{b\pi}{\th}\left(\f{Q_0}{\sqrt{\rho_-}}\right)^{\f\pi\th-1},$$  due to the conformal mapping in
\cite{MI}.
\end{remark}

As we mentioned before, A. Friedman established the well-posedness
results for the simple case of horizontal blade surfaces ($b=0$),
and proposed an open problem on the general case as shown in Figure
\ref{f3}. However, from the mathematical point of view, the
extension to the present problem is not straightforward, and
involves some additional difficulties. For the special case $b=0$
(see Figure \ref{f2}), consider the critical case $Q_0=0$, the
injection flow vanishes and the mainstream flow is nothing but a
trivial uniform flow. The free boundary is the segment connecting
the leading edge $A$ and the trailing edge $B$. However, for the
general case ($b\neq0$), there does not exist a trivial flow for the
critical case $Q_0=0$. This is the one of main differences and the
difficulties here. This fact prevents us to establish the lower
bound of $\ld$ while $Q_0\rightarrow 0$. To overcome this difficulty,
we will investigate the limiting flow ($Q_0\rightarrow0$), and show
that the free boundary initiates smoothly at $A$ and terminated at
the wall $S_2$. Moreover, we will show that the intersection of the
free boundary $\Gamma$ and $S_2$ must below the trailing edge $B$.
Another difference is that the domain is a star-sharped one with
respect to $B$ for the special case $b=0$, we can take a rescaling
transform to obtain the uniqueness. Furthermore, the property can
not hold for the general case, and we have to develop a new method
to obtain the uniqueness.

The basic idea in this paper is to seek a two-phase fluid with a
smooth interface connecting at the leading edge $A$. A truncated
injection flow problem is presented in Section 2, and furthermore,
we give a result on existence and uniqueness in truncated fluid
field. Section 3 studies some useful properties of the minimizer and
free boundary in the truncated domain. In particular, we will
establish the relationship between the jump constant $\ld$ and the
injected flux $Q_0$, which builds a bridge between the injection flow
problem 1 and 2. Section 4 is devoted to the solution of the
injection flow problem using some uniform estimates of the solution
in truncated domain. The analysis reveals the existence and
uniqueness of the two-phase fluid with $C^1$-smooth interface, the
fact firstly proved in \cite{FA3} for special case. Our results
solve the open problem on the well-posedness of an ideal fluid
injected obliquely from a slot into a stream.

\section{The truncated injection flow problem}
To solve the injection flow problem, we first study the truncated
injection flow problem with finite height in this section. To
simplify notation, denote
$$Q=\f{Q_0}{\sqrt{\rho_-}}.$$

For any $L>b$, denote
$$N_L=\{(x,L)\mid -\infty<x<+\infty\}\ \ \text{and}\ \ \O_L=\O\cap\{y<L\}.\ \ \text{(See Figure \ref{f9})}$$

\begin{figure}[!h]
\includegraphics[width=100mm]{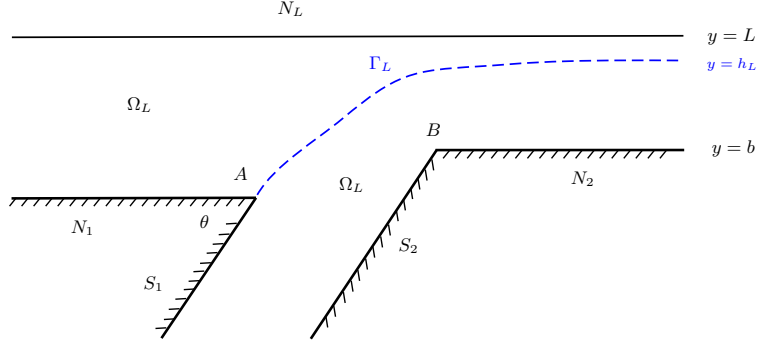}
\caption{Truncated flow field}\label{f9}
\end{figure}

The definition of the truncated injection flow problem will be given
in the following.

The truncated injection flow problem 1 corresponding to the
injection flow problem 1 is as follows: For any given $Q>0$, does
there exist a unique $\ld_L$ and a unique injection flow
$(\psi_L,\Gamma_L)$ in the truncated domain $\O_L$, such that the
mainstream flow possesses uniform speed in upstream, and the
interface $\Gamma_L$ connects at $A$ and extends to infinity?

Next, we will give the definition of the solution to the truncated
injection flow problem 1.

\begin{definition}\label{def2}
{\bf(A solution to the truncated injection flow problem 1).}\\
For any $L>b$, a vector $(\psi_L,\Gamma_L)$ is called a solution to
the truncated injection flow
problem 1, provided that \\
(1) $\Delta\psi_L=0$ in $\O_L\setminus \Gamma_L$, $\psi_L\in
C^0(\bar\O_L)$ and $\g\psi_L\in
L^{\infty}(\O_L\setminus B_\e(B))$ for any $\e>0$.\\
(2) $\psi_L=0$ satisfies the Dirichlet boundary conditions \eqref{a0} and $\psi_L=L$ on $N_L$. \\
 (3) The free
boundary $\Gamma_{L}: y=k_L(x)$, and $k_L(x)$ is a $C^1$-smooth
 strictly increasing function in $(0,+\infty)$, and $k_L(x)>b$
 for any $x\geq a$. Furthermore,
\be\label{b1}k_L(0)=0,\ee and there exists a $h_L\in(b,L)$, such
that $$ \lim_{x\rightarrow+\infty}k_L(x)=h_L\ \ \text{and}\ \
\lim_{x\rightarrow+\infty}k_L'(x)=0.$$  \\
(4) $\psi_L$ satisfies the Rankine-Hugoniot jump condition on
$\Gamma_L$, namely,
\be\label{b00}\left(\f{\p\psi_L^-}{\p\nu}\right)^2-\left(\f{\p\psi_L^+}{\p\nu}\right)^2=\ld_L\
\ \text{on}\ \ \Gamma_L,\ee
where $\ld_L=\f{Q^2}{(h_L-b)^2}-\f{L^2}{(L-h_L)^2}$.\\
(5) $\Gamma_L$ is continuously differentiable at $A$ and
\be\label{b2}k_L'(0+0)=\left\{\ba{ll} \tan\th,\
&\text{if $\ld_L>0$},\\
0,\
&\text{if $\ld_L<0$},\\
\tan\f{\th}2,\ &\text{if}\ \ld_L=0.\ea\right.\ee\\
 (6) $\psi_L$ has the
following asymptotic behaviors
$$\psi_L(x,y)\rightarrow\left\{\ba{ll} \f{Q(y-h_L)}{h_L-b},\
&\text{if $b<y<h_L$, as $x\rightarrow+\infty$},\\
\f{L(y-h_L)}{L-h_L},\ &\text{if $h_L<y<L$, as
$x\rightarrow+\infty$},\ea\right.$$
$$\psi_L(x,y)\rightarrow y,\quad \text{if $0<y<L$, as
$x\rightarrow-\infty$},$$ and
$$\left|\psi_L(x,y)-\f{Q(y\cos\th-x\sin\th)}{a\sin\th-b\cos\th}\right|\ra
0 \ \ \text{uniformly in any compact subset of $S$,}$$ as
$y\ra-\infty$, where $S=\{(x,y)\mid
y\cot\th<x<(y-b)\cot\th+a,-\infty<y<+\infty\}$.\\
(7) $\f{L(y-h_L)}{L-h_L}\leq\psi_L^+(x,y)\leq y$ in
$\O_L\cap\{y>0\}$.

\end{definition}

\begin{remark}\label{re2}
It should be noted that $f(t)=\f{Q^2}{(t-b)^2}-\f{L^2}{(L-t)^2}$ is
a strictly monotone decreasing function for $t\in(b,L)$. Therefore,
the asymptotic height $h_L\in(b,L)$ of the free boundary can be
determined uniquely by $\ld=\f{Q^2}{(h_L-b)^2}-\f{L^2}{(L-h_L)^2}$.
\end{remark}

\subsection{Variational approach}
To solve the truncated injection flow problem 1, as the first step,
we introduce a truncated variational problem for any given parameter
$\ld\in(-\infty,+\infty)$ and $Q>0$. Secondly, we will verify that
there exists a unique parameter $\ld=\ld_L$, such that the interface
$\Gamma_L$ connects at the leading edge $A$. Finally, taking
$L\rightarrow+\infty$ yields the existence of solution to the
injection flow problem 1.

For any $\mu>1$, denote $$\O_{L,\mu}=\O_L\cap\{(x,y)\mid x>-\mu,
y>-\mu\}\ \ \text{and}\ \ \sigma_{L,\mu}=\{(-\mu,y)\mid 0\leq y\leq
L\},$$ and
 $$D_{1,L,\mu}=\O_{L,\mu}\cap\{(x,y)\mid
x<0,y<0\}\ \ \text{and}\ \ D_{2,L,\mu}=\O_{L,\mu}\cap\{(x,y)\mid
x<0, y>0\},$$ and  $$N_{1,\mu}=N_1\cap\{x\geq-\mu\},\ \
S_{1,\mu}=S_1\cap\{y\geq-\mu\}\ \ \text{and}\ \ N_{L,\mu}=N_L\cap\{
x\geq-\mu\},$$ and
 $$S_{\mu}=\{-\mu\tan\th<x<a-(\mu+b)\tan\th, y=-\mu\}\ \ \text{and}\ \ S_{2,\mu}=S_2\cap\{
y\geq-\mu\}.$$ Please see Figure \ref{f10}.

\begin{figure}[!h]
\includegraphics[width=100mm]{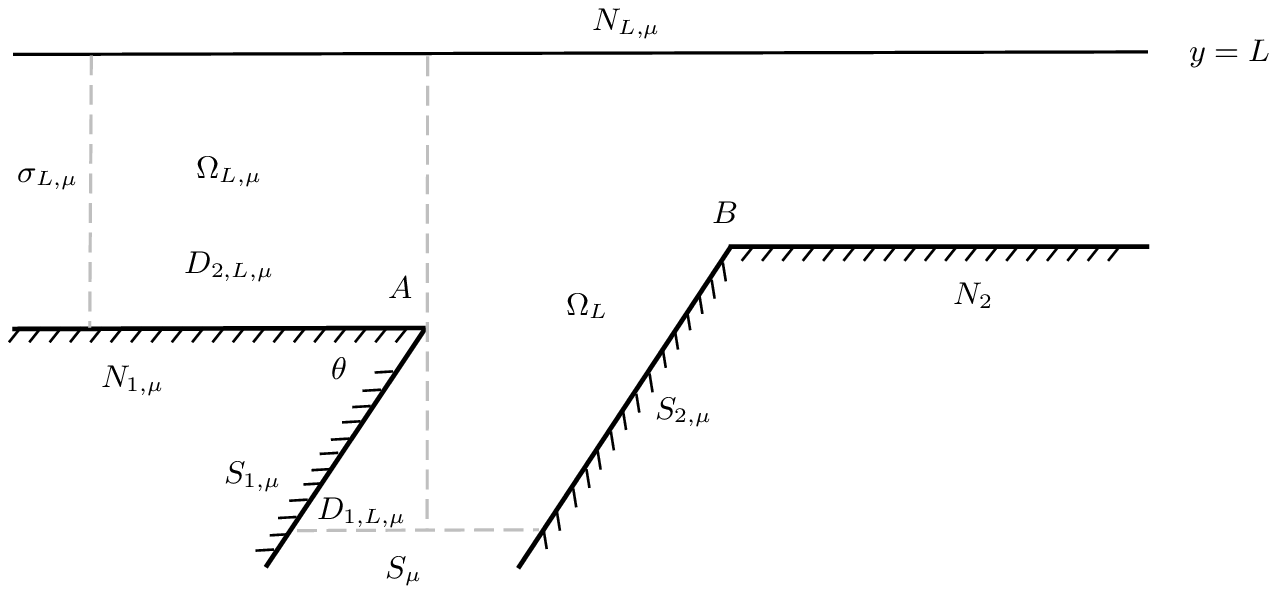}
\caption{The nozzle walls and the ground}\label{f10}
\end{figure}

As mentioned in Remark \ref{re2}, for any given $\ld$ and $L>b$, we
can obtain a unique asymptotic height $h_L\in(b,L)$ of the interface
$\Gamma_{\ld,L}$. Then, we can define $\ld_{1,L}$, $\ld_{2,L}$ and
$\ld_{0,L}$ as follows
$$\ld_{1,L}=\f{Q}{h_L-b},\ \
\ld_{2,L}=\f{L}{L-h_L}\ \  \text{and}\ \
\ld_{0,L}=\min\{\ld_{1,L},\ld_{2,L}\}.$$ Obviously, $
\ld=\ld_{1,L}^2-\ld_{2,L}^2$. Moreover, we give the following
functional
$$J_{\ld,L,\mu}(\psi)=\int_{\O_{L,\mu}}\left|\nabla\psi-(\ld_{1,L}I_{\{\psi<0\}}+\ld_{2,L}I_{\{\psi>0\}}+\ld_{0,L}I_{\{\psi=0\}})I_{\{x>0\}}e\right|^2dxdy$$
where $e=(0,1)$ and $I_{D}$ is the characteristic function of the
set $D$. And the admissible set is defined as follows
$$\ba{rl}K_{L,\mu}=\{\psi\mid &\psi\in H^1_{loc}(\O_{L,\mu}),\ -Q\leq\psi\leq L\ \text{a.e. in $\O_{L,\mu}$},\ \ 0\leq\psi\leq L\
\text{a.e. in $D_{2,L,\mu}$},\\
&-Q\leq\psi\leq 0\ \text{a.e. in $D_{1,L,\mu}$},\ \psi=-Q\ \text{on
$S_\mu\cup S_{2,\mu}\cup N_2$},\\
&\ \psi=L\ \text{on $N_{L,\mu}\cup\sigma_{L,\mu}$},\
\psi=\max\{L-(x+\mu)L,0\}\ \text{on
$N_{1,\mu}$},\\
&\ \psi=\min\{-Q+(y+\mu)Q,0\}\ \text{on $S_{1,\mu}$}\}.\ea$$

{\bf The truncated variational problem $(P_{\ld,L,\mu})$:} For any
$L>b$, $\mu>1$ and $\ld\in(-\infty,+\infty)$, find a
$\psi_{\ld,L,\mu}\in K_{L,\mu}$ such that
$$J_{\lambda,L,\mu}(\psi_{\ld,L,\mu})=\min_{\psi\in K_{L,\mu}}
J_{\lambda,L,\mu}(\psi).$$ Define the free boundary in the truncated
domain as
$$\Gamma_{\ld,L,\mu}=\O_{L,\mu}\cap\{x>0\}\cap\{\psi_{\ld,L,\mu}=0\}.$$

For any $L>b$, $\mu>1$ and $\ld\in\mathbb{R}$, the existence and
uniqueness of the minimizer to the truncated variational problem
$(P_{\ld,L,\mu})$ can be established along the proof of Theorem 2.1
and Lemma 2.2 in \cite{FA3}. We state the results in the following.
\begin{prop}\label{lb1}(Theorem 2.1 and Lemma 2.2 in \cite{FA3}) For any $L>b$, $\mu>1$ and $\ld\in(-\infty,+\infty)$,
there exists a unique minimizer $\psi_{\ld,L,\mu}$ to the truncated
variational problem $(P_{\ld,L,\mu})$. Moreover,
$$\ba{rl}\Gamma_{\ld,L,\mu}=&\O_{L,\mu}\cap\{x>0\}\cap\{\psi_{\ld,L,\mu}=0\}\\
=&\O_{L,\mu}\cap\{x>0\}\cap\p\{\psi_{\ld,L,\mu}<0\}\\
=&\O_{L,\mu}\cap\{x>0\}\cap\p\{\psi_{\ld,L,\mu}>0\},\ea$$ and
$\psi_{\ld,L,\mu}(x,y)$ is monotone increasing with respect to $y$
and there exists a continuous function $k_{\ld,L,\mu}(x)$ for $x>0$,
such that
$$\Gamma_{\ld,L,\mu}=\{(x,y)\in \O_L\mid x>0, y=k_{\ld,L,\mu}(x)\}.$$
$\psi_{\ld,L,\mu}$ satisfies the free boundary condition in the weak
sense, namely, \be\label{b0}\ba{rl}\lim_{\e\rightarrow
0^+,\delta\rightarrow
0^+}&\left(\int_{\O_{L,\mu}\cap\{x>0\}\cap\p\{\psi_{\ld,L,\mu}>\e\}}(|\g\psi_{\ld,L,\mu}|^2-\ld_{2,L}^2)\eta\cdot\nu
dS\right.
\\
&\left.+\int_{\O_{L,\mu}\cap\{x>0\}\cap\p\{\psi_{\ld,L,\mu}<-\delta\}}(|\g\psi_{\ld,L,\mu}|^2-\ld_{1,L}^2)\eta\cdot\nu
dS\right)=0.\ea\ee Furthermore, if $\ld<0$ and $|\ld|$ is
sufficiently
large, then we have\\
(1) $\psi_{\ld,L,\mu}(x,y)$ is monotone decreasing with respect to
$x$.\\
(2) $k_{\ld,L,\mu}(x)$ is monotone increasing with respect to
$x>0$.\\
(3) $k_{\ld,L,\mu}(0)=\lim_{x\ra 0^+}k_{\ld,L,\mu}(x)$ exists and
$0\leq k_{\ld,L,\mu}(x)\leq L$.
\end{prop}

\subsection{The regularity of the free boundary $\Gamma_{\ld,L,\mu}$}

In Theorem 8.12 in \cite{ACF4}, Alt, Caffarelli and Friedman proved
that the free boundary $\Gamma_{\ld,L,\mu}$ of the minimizer
$\psi_{\ld,L,\mu}$ is $C^1$-smooth. Based on the significant work
\cite{C1} by Caffarelli, we will obtain the higher regularity of the
free boundary of the minimizer in this subsection. First, we give
the definition of the weak solution of a free boundary problem as in
\cite{C1}.

\begin{definition}\label{def3} Assume that $G(t)$ is a continuous strictly
monotone increasing function with respect to $t\in\mathbb{R}$, which
satisfies that $G(t)\geq t$ and $t^{-C}G(t)$ is decreasing with
respect to $t> 0$, for some large $C>0$. Let $E$ be a bounded open
set in $\O_{L,\mu}\cap\{x>0\}$. A continuous function $\psi$ in $E$
is called a {\it weak solution} of the free boundary problem,
provided that $\psi$ satisfies

 (1) $\Delta\psi=0$ in $E^+(\psi)=E\cap\{\psi>0\}$,

(2) $\Delta\psi=0$ in $E^-(\psi)=\text{int}(E\cap\{\psi\leq 0\})$,

(3) (The weak free boundary condition) $\psi$ satisfies the free
boundary condition
$$\psi^+_{\nu}=G(\psi_\nu^-)\ \ \text{along}\ \
\mathcal{F}(\psi)=E\cap\p\{\psi>0\},$$ in the following sense.

For any $X_0\in\mathcal{F}(\psi)$, if $\mathcal{F}(\psi)$ has an
one-side tangent ball at $X_0$ (i.e., there exists a ball $B_r(Y)$,
such that $X_0\in\p B_r(Y)$ and $B_r(X)$ is contained either in
$E^+(\psi)$ or in $E^-(\psi)$), then
$$\psi(X)=\alpha<X-X_0,\nu>^+-\beta<X-X_0,\nu>^-+o(|X-X_0|),\ \ \text{ $\beta\geq0$ and $\alpha=G(\beta)$},$$
where $\nu$ is the unit radial direction of $\p B_r(Y)$ at $X_0$
pointing into $E^+(\psi)$, $<p,q>^+=\max\{p\cdot q,0\}$ and
$<p,q>^-=\max\{-p\cdot q,0\}$.

\end{definition}

Next, we will obtain the analyticity of the free boundary
$\Gamma_{\ld,L,\mu}$ in the following, which implies that the
Rankine-Hugoniot condition \eqref{b00} on the free boundary holds in
the classical sense. The main idea borrows from the works
\cite{A1,C1,MB}.
\begin{theorem}\label{lb3}The free boundary $\Gamma_{\ld,L,\mu}$ is analytic.\end{theorem}

\begin{proof}
{\bf Step 1.} In this step, we will show that the minimizer
$\psi_{\ld,L,\mu}$ to the truncated variational problem
$(P_{\ld,L,\mu})$ is a weak solution in Definition \ref{def3}.

Similar to Theorem 2.2 in \cite{ACF4}, it is easy to verify that the
minimizer $\psi_{\ld,L,\mu}$ is harmonic in
$E\setminus\{\psi_{\ld,L,\mu}=0\}$, where $E$ is a bounded open set
in $\O_{L,\mu}\cap\{x>0\}$, which implies that $\psi_{\ld,L,\mu}$
satisfies the conditions (1) and (2) in Definition \ref{def3}. Next,
it suffices to verify the condition (3) in Definition \ref{def3}.
Without loss of generality, we assume that $\ld\leq0$.

For any $X_0\in\mathcal{F}(\psi_{\ld,L,\mu})$, by means of Theorem
7.4 in \cite{ACF4}, we have
 $$\psi_{\ld,L,\mu}(X)=\alpha<X-X_0,\nu>^+-\beta<X-X_0,\nu>^-+o(|X-X_0|),$$
 where $\alpha>0$, $\beta>0$ and $\ld=\beta^2-\alpha^2$.

 Thus,
$\alpha=G(\beta)=\sqrt{\beta^2-\ld}$. It is easy to see that
$G(\beta)$ is strictly monotone increasing with respect to $\beta$
and $\beta^{-1}G(\beta)$ is decreasing with respect to $\beta$.

Hence, we conclude that the minimizer $\psi_{\ld,L,\mu}$ is a weak
solution in Definition \ref{def3}.

{\bf Step 2.} Next, we will obtain the analyticity of the free
boundary.

 Since $\psi_{\ld,L,\mu}$ is the weak solution in
Definition \ref{def3}, by using Theorem 1 in \cite{C1}, we can
conclude that the free boundary
$\O_{L,\mu}\cap\p\{\psi_{\ld,L,\mu}>0\}$ is $C^{1,\alpha}$ for some
$\alpha\in(0,1)$.

Denote $\psi=\psi_{\ld,L,\mu}$ for simplicity. Since
$$\Gamma_{\ld,L,\mu}=\O_{L,\mu}\cap\{x>0\}\cap\p\{\psi>0\}=\O_{L,\mu}\cap\{x>0\}\cap\{\psi=0\},$$
is $C^{1,\alpha}$-smooth, which implies that the
$\mathcal{L}^2$-measure of the free boundary $\Gamma_{\ld,L,\mu}$ is
zero. Therefore, we can use a $C^{1,\alpha}$ transformation to
flatten the free boundary. Then reflect $\psi^+$ to the full
neighborhood of the free boundary, applying the Schauder estimates
for elliptic equation in divergence form in Section 9 in \cite{A1},
we can obtain the $C^{1,\alpha}$ regularity of $\psi^+$ up to the
free boundary. Similarly, we can obtain the $C^{1,\alpha}$
regularity of $\psi^-$ up to the free boundary. Moreover, it follows
from \eqref{b0} that \be\label{b3}|\g\psi^-|^2-|\g\psi^+|^2=\ld\ \
\text{on the free boundary}.\ee

If $\ld=0$, it follows from Theorem 2.2 in \cite{ACF4} that $\psi$
is harmonic in $\O_{L,\mu}\cap\{x>0\}$. By means of the monotonicity
of $\psi(x,y)$ with respect to $y$, the strong maximum principle
gives that $\p_y\psi>0$ in $\O_{L,\mu}\cap\{x>0\}$. Hence, the
implicit function theorem gives that the level set
$\O_{L,\mu}\cap\{x>0\}\cap\{\psi=0\}$ is analytic.

If $\ld\neq0$, without loss of generality, we assume that $0$ is a
free boundary point of $\psi$, $|\g\psi^+(0)|\neq 0$ and the inner
normal to $\Gamma_{\ld,L,\mu}$ at $0$ is in the direction of the
positive $y$-axis. Extend $\t\psi$ as a $C^{1,\alpha}$ function into
a full neighborhood of $0\in\Gamma_{\ld,L,\mu}$, such that
$\t\psi=\psi$ in $\{\psi>0\}$. In view of $|\g\psi^+(0)|\neq 0$, one
has \be\label{b4}\t\psi_y(0)>0.\ee Define a mapping as follows,
$$S=TX=(s,t)\triangleq (x,\t\psi(x,y)), \ \ X=(x,y).$$ By virtue of \eqref{b4}, it is easy to
check that
$$\text{det}\left(\f{\p S}{\p X}\right)=\t\psi_y(x,y)>0\ \ \text{in a small neighborhood of $0$}.$$
And thus the mapping $T$ is a local diffeomorphism near $0$.

Denote the inverse transform as
\be\label{b5}\left\{\ba{rl}&x=s\\
&y=\phi(s,t).\ea\right.\ee

Therefore, the free boundary $\Gamma_{\ld,L,\mu}$ is transformed
into $t=0$, and we have
$$\left(\f{\p X}{\p S}\right)=\left(\f{\p S}{\p X}\right)^{-1}=\left(\begin{matrix} 1 &0 \\
-\f{\t\psi_x}{\t\psi_y} &\f{1}{\t\psi_y}
\end{matrix}\right).$$
Consequently, one has
$$\phi_s=\f{\p y}{\p s}=-\f{\t\psi_x}{\t\psi_y},\ \ \phi_t=\f{\p y}{\p
t}=\f{1}{\t\psi_y},$$ and \be\label{b6}\f{\p t}{\p
x}=\t\psi_x=-\f{\phi_s}{\phi_t},\ \ \f{\p t}{\p
y}=\t\psi_y=\f{1}{\phi_t}.\ee It follows from \eqref{b6} that
$$\mathcal{Q}\phi=\p_s\left(\f{\phi_s}{\phi_t}\right)+\p_t\left(-\f{1+\phi^2_s}{2\phi^2_t}\right)=0\ \ \text{in the neighborhood of $0$}.$$
Denote $A_1(\phi_s,\phi_t)=\f{\phi_s}{\phi_t}$ and
$A_2(\phi_s,\phi_t)=-\f{1+\phi^2_s}{2\phi^2_t}$.
It is easy to check that $$\mathcal{A}=\f{\p(A_1,A_2)}{\p(\phi_s,\phi_t)}=\f{1}{\phi_t^3}\left(\begin{matrix} \phi_t^2 &-\phi_s\phi_t \\
-\phi_s\phi_t &1+\phi_s^2
\end{matrix}\right).$$
Direct straightforward computations give that the matrix
$\mathcal{A}$ has two eigenvalues
$$\ld_1=\f{1+\phi_s^2+\phi_t^2+\sqrt{((1+\phi_t)^2+\phi_s^2)((1-\phi_t)^2+\phi_s^2)}}{2\phi_t^3},$$ and
$$\ld_2=\f{2}{\left(1+\phi_s^2+\phi_t^2+\sqrt{((1+\phi_t)^2+\phi_s^2)((1-\phi_t)^2+\phi_s^2)}\right)\phi_t}.$$
In view of \eqref{b4}, one has
$$\ld_1>0\ \ \text{and}\ \ \ld_2>0\ \ \text{in a small neighborhood of $0$}.$$

Thus, $\mathcal{Q}\phi=0$ is a quasilinear elliptic equation in a
neighborhood $E$ of $0$. Furthermore, $\phi$ satisfies the Neumann
type boundary condition as follows,
\be\label{b7}\left\{\ba{ll}\mathcal{Q}\phi=0~~~~&\text{in}\ \ \ D,\\
\f{\phi_t}{\sqrt{1+\phi_s^2}}=g(s)\ \ &\text{on}\ \ \bar
D\cap\{t=0\},\ea\right.\ee where $D=E\cap\{t>0\}$ and
$g(s,0)=\f{1}{\sqrt{|\g\psi^-(s,0)|^2-\ld}}$.

 Noting that $\t\psi$ is in $C^{1,\alpha}$ near $0$, we have that the coefficients of
 $\mathcal{Q}$ and $g(s)$ are $C^{0,\alpha}$.
 By using the elliptic regularity in Section 9 in \cite{A1}, we obtain that
 $\phi(S)$ is $C^{2,\alpha}$ near $0$. Furthermore, the free boundary $\Gamma_{\ld,L,\mu}$ can be described by
 $y=\phi(s,0)=\phi(x,0)$, and thus the free boundary $\Gamma_{\ld,L,\mu}$ is $C^{2,\alpha}$ near $0$.

Applying the Schauder estimates for elliptic equations in \cite{A1},
we can obtain the $C^{2,\alpha}$ regularity of $\psi^+$ and $\psi^-$
up to the free boundary $\Gamma_{\ld,L,\mu}$. By using the above
arguments, we can conclude that the free boundary
$\Gamma_{\ld,L,\mu}$ is $C^{3,\alpha}$.

Along the bootstrap arguments, the $C^{\infty}$ regularity of the
free boundary $\Gamma_{\ld,L,\mu}$ can be established. Finally, with
the aid of the results in Section 6.7 in \cite{MB}, we can obtain
that $\phi(s,t)$ is analytic in $t\geq 0$. Hence, we obtain the
analyticity of the free boundary $\Gamma_{\ld,L,\mu}$.

\end{proof}



It follows from Lemma 2.4 and Lemma 2.5 in \cite{FA3} that we can
obtain the existence of $\ld_{L,\mu}$, such that the continuous fit
condition of $\Gamma_{\ld_{L,\mu},L,\mu}$ at $A$ holds for any $L>b$
and $\mu>1$. Similar to the arguments on the compressible subsonic
flows in infinitely long nozzle in \cite{DD, DXX, DXY, XX1, XX2,
XX3}, we can obtain the asymptotic behavior of the flows in
downstream and in upstream. We omit the details here.
\begin{prop}\label{lb5}For any $Q>0$, $L>b$ and $\mu>1$, there exists a $\ld_{L,\mu}\in\mathbb{R}$ with $|\ld_{L,\mu}|\leq C$
 (the constant $C$ is independent of $\mu$), such that $k_{\ld_{L,\mu},L,\mu}(0)=0$ and
$\psi_{\ld_{L,\mu},L,\mu}$ is Lipschitz continuous in
$\bar\O_{L,\mu}\setminus B_\e(B)$ for any small $\e>0$. Furthermore,
$$\text{$k_{\ld_{L,\mu},L,\mu}(x)\ra h_{\ld_{L,\mu}}\in(b,L)$ and
$\psi_{\ld_{L,\mu},L,\mu}(x,y)\ra\psi_0(y)$}$$ in any compact subset
of $(b,L)$ as $x\ra+\infty$, where $h_{\ld_L,\mu}$ is determined
uniquely by
$$\ld_{L,\mu}=\f{Q^2}{(h_{\ld_{L,\mu}}-b)^2}-\f{L^2}{(L-h_{\ld_{L,\mu}})^2}\ \ \text{and}\ \ \psi_0(y)=\left\{\ba{ll} \f{Q(y-h_{\ld_{L,\mu}})}{h_{\ld_{L,\mu}}-b},\
&\text{if $b<y<h_{\ld_{L,\mu}}$},\\
\f{L(y-h_{\ld_{L,\mu}})}{L-h_{\ld_{L,\mu}}},\ &\text{if
$h_{\ld_{L,\mu}}<y<L$},\ea\right.$$

\end{prop}

\subsection{The existence of the truncated injection flow problem 1}
In this subsection, we will investigate the existence of the
truncated injection flow problem 1. Moreover, the positivity of
horizontal velocity and vertical velocity will be obtained.

\begin{theorem}\label{lb6} For any $Q>0$ and $L>b$, there exist a $\ld_L$ and a solution $(\psi_{\ld_L},\Gamma_{\ld_L})$ to the truncated injection flow
problem 1.

\end{theorem}
\begin{proof} Since $\psi_{\ld_{L,\mu},L,\mu}$ is Lipschitz continuous in $\O_{L,\mu}$ and $|\ld_{L,\mu}|\leq C$( the constant $C>0$ is independent of $\mu$),
 we can take a sequence $\{\mu_n\}$ with $\mu_n\rightarrow+\infty$, then there exist a $\ld_L$ and a $\psi_{\ld_L,L}\in H_{loc}^1(\O_L)$, such that $$\ld_{L,\mu_n}\rightarrow\ld_L,$$ and
$$\psi_{\ld_{L,\mu_n},L,\mu_n}\rightarrow\psi_{\ld_L,L}\ \text{in
$H^1_{loc}(\O_L)$ and uniformly in any compact subset of $\O_L$},$$
as $\mu_n\ra+\infty$.

Next, we divide six steps to verify that
$(\psi_{\ld_L},\Gamma_{\ld_L})$ satisfies the conditions in
Definition \ref{def2}.

{\bf Step 1.} For any $Q>0$ and $L>b$, denote  $$D_{1,L}=\O_{L}\cap\{(x,y)\mid
x<0,y<0\}\ \ \text{and}\ \ D_{2,L}=\O_{L}\cap\{(x,y)\mid
x<0, y>0\}.$$
By virtue of Lemma 6.2 in \cite{ACF4}, we can show
that $\psi_{\ld_L,L}$ is a local minimizer to the variational
problem $(P_{\ld_L,L})$, namely,
$$P_{\ld_L,L}:\ \ J_{D}(\psi_{\ld_L,L})=\min J_{D}(\psi)\ \ \text{for any $\psi\in K_L$ and
$\psi=\psi_{\ld_L,L}$ on $\p D$},$$ where
$$J_{D}(\psi)=\int_{D}\left|\nabla\psi-(\ld_{1,L}I_{\{\psi<0\}}+\ld_{2,L}I_{\{\psi>0\}}+\ld_{0,L}I_{\{\psi=0\}})I_{\{x>0\}}e\right|^2dxdy$$ and $$\ba{rl}K_{L}=\{\psi\mid &\psi\in H^1_{loc}(\O_{L}),\ -Q\leq\psi\leq L\ \text{a.e. in $\O_{L}$},\ \ 0\leq\psi\leq L\
\text{a.e. in $D_{2,L}$},\\
&-Q\leq\psi\leq 0\ \text{a.e. in $D_{1,L}$},\ \psi=-Q\ \text{on
$S_{2}\cup N_2$},\\
&\ \psi=L\ \text{on $N_{L}$},\
\psi=0\ \text{on
$N_{1}\cup S_1$}\}.\ea$$
for any bounded domain $D\subset \O_L$. Therefore, the conditions
(1) and (2) in Definition \ref{def2} have be verified.

{\bf Step 2.} We can conclude that $\psi_{\ld_L,L}(x,y)$ is monotone
increasing with respect to $y$ and decreasing with respect to $x$,
which follows from the monotonicity of
$\psi_{\ld_{L,\mu},L,\mu}(x,y)$. Furthermore, the free boundary
$$\Gamma_{\ld_L,L}= \O_L\cap\{x>0\}\cap\{\psi_{\ld_L,L}=0\}$$  of the minimizer $\psi_{\ld_L,L}$ is
given by a continuous function $y=k_{\ld_L,L}(x)$ for any $x>0$. In
particular, $k_{\ld_L,L}(0)=0$. In view of Proposition \ref{lb5},
one has \be\label{b10}\text{$k_{\ld_{L},L}(x)\ra h_{L}$ and
$\psi_{\ld_{L},L}(x,y)\ra\psi_0(y)$}\ee in any compact subset of
$(b,L)$ as $x\ra+\infty$, where  $h_{L}$ is determined uniquely by
$$\ld_{L}=\f{Q^2}{(h_{L}-b)^2}-\f{L^2}{(L-h_{L})^2}\ \ \text{and}\ \ \psi_0(x,y)=\left\{\ba{ll} \f{Q(y-h_{L})}{h_{L}-b},\
&\text{if $b<y<h_{L}$},\\
\f{L(y-h_{L})}{L-h_{L}},\ &\text{if $h_{L}<y<L$}.\ea\right.$$
Furthermore, by virtue of the analyticity of the free boundary
$\Gamma_{\ld_L,L}$ and \eqref{b10}, one has
$$k'_{\ld_L,L}(x)\ra 0\ \ \text{as}\ \ x\ra+\infty.$$ Next, we will
show that $k_{\ld_L,L}(x)$ is strictly monotone increasing with
respect to $x>0$. If not, there exist $x_1,x_2\in(0,+\infty)$ with
$x_1<x_2$, such that $k_{\ld_L,L}(x_1)=k_{\ld_L,L}(x_2)$. The
monotonicity of $\psi_{\ld_L,L}(x,y)$ with respect to $x$ and $y$
gives that there exists a small $r>0$, such that
$$\text{$\psi_{\ld_L,L}>0$ in $B_r(X_0)\cap\{y>y_0\}$ and
$\psi_{\ld_L,L}<0$ in $B_r(X_0)\cap\{y<y_0\}$,}$$ where
$X_0=(x_0,y_0)$ with $x_0=\f{x_1+x_2}2$ and $y_0=k_{\ld_L,L}(x_1)$.
Denote $I_0=\{(x,y_0)\mid -r<x-x_0<r\}$,
$B^+=B_{2r}(X_0)\cap\{y>y_0\}$ and $B^-=B_{2r}(X_0)\cap\{y<y_0\}$.
Set $\psi_1=\psi_{\ld_L,L}$ in $B^-$ and $\psi_2=\psi_{\ld_L,L}$ in
$B^+$. Since $\psi_{\ld_L,L}(x,y)$ is decreasing with respect to
$x$, the strong maximum principle gives that
$$\p_x\psi_1<0\ \ \text{in $B^-$ and }\ \ \p_x\psi_2<0\ \ \text{in $B^+$}.$$ In view of that $\p_x\psi_1=\p_x\psi_2=0$ on $I_0$,
it follows from Hopf's lemma that \be\label{bb11}\f{\p}{\p
y}\p_x\psi_1>0\ \ \text{and }\ \ \f{\p}{\p y}\p_x\psi_2<0\ \
\text{on $I_0$}.\ee Noting that $\psi_1=\psi_2=0$ on $I_0$, thanks
hopf's lemma, one has
$$\p_y\psi_1>0\ \ \text{and }\ \ \p_y\psi_2>0\ \ \text{on $I_0$},$$
which together with \eqref{bb11} give that
$$\f{\p(\p_{y}\psi_1)^2}{\p
x}>0\ \ \text{and }\ \ \f{\p(\p_y\psi_2)^2}{\p x}<0\ \ \text{on
$I_0$}.
$$ Thus one has
$$\f{\p}{\p
x}\left(|\g\psi^-_{\ld_L,L}|^2-|\g\psi^+_{\ld_L,L}|^2\right)=\f{\p}{\p
x}\left((\p_{y}\psi_1)^2-(\p_{y}\psi_2)^2\right)>0\ \ \text{on
$I_0$},
$$ which contradicts to the fact $|\g\psi^-_{\ld_L,L}|^2-|\g\psi^+_{\ld_L,L}|^2=\ld_L$
on $I_0$.

{\bf Step 3}. It follows from \eqref{b0} that
$$\ba{rl}\lim_{\e\rightarrow 0^+,\delta\rightarrow
0^+}&\left(\int_{\O_{L}\cap\{x>0\}\cap\p\{\psi_{\ld_L,L}>\e\}}(|\g\psi_{\ld_L,L}|^2-\ld_{2,L}^2)\eta\cdot\nu
dS\right.
\\
&\left.+\int_{\O_{L}\cap\{x>0\}\cap\p\{\psi_{\ld_L,L}<-\delta\}}(|\g\psi_{\ld_L,L}|^2-\ld_{1,L}^2)\eta\cdot\nu
dS\right)=0.\ea$$ Since the free boundary $\Gamma_{\ld_L,L}$ is
analytic, $\psi_{\ld_L,L}$ is $C^{2}$ up to the free boundary
$\Gamma_{\ld_L,L}$. Thus, the condition (4) in Definition \ref{def2}
holds.

{\bf Step 4.} In this step, we will show that the free boundary
$\Gamma_{\ld_L,L}$ is continuous differentiable at $A$, and
\eqref{b2} holds.

For any $r>0$, define a blow-up sequence $\{\psi_n\}$, such that
$\psi_n(\t X)=\f{\psi_{\ld_L,L}(r_n\t X)}{r_n}$ in $B_{r}(0)$ with
$\t X=(\t x,\t y)$ and $r_n\rightarrow 0$. We next consider the
following three cases.

{\bf Case 1}. $\ld_L<0$. Taking a sequence $\{X_n\}$ with
$X_n\in\Gamma_{\ld_L,L}$ and $X_n\rightarrow 0$, set $r_n=|X_n|$. By
virtue of the non-degeneracy Theorem 3.1 in \cite{ACF4}, one has
\be\label{b11}\f1{r_n}\fint_{\p
B_{\f{r_n}2}(X_n)}\psi_{\ld_L,L}^+(X)dS_X\geq c|\ld_L|^{\f12},\ee
where $c>0$ is a constant independent of $n$. Denote $\t
Y_n=\f{X_n}{r_n}$, it follows from \eqref{b11} that
\be\label{b12}\fint_{\p B_{\f12}(\t Y_n)}\psi_n^+(\t X)dS_{\t X}\geq
c|\ld_L|^{\f12}\ \ \text{and}\ \ |\t Y_n|=1.\ee

It follows from the similar arguments in Pages 444-445 in
\cite{ACF4} that there exist a subsequence $\{\psi_n\}$ and a
blow-up limit $\psi_0\in H_{loc}^1(\mathbb{R}^2)$, such that
\be\label{b13}\psi_n(\t X)\rightarrow\psi_0(\t X)\ \text{uniformly
in bounded sets, and $\g\psi_n\rightarrow\g\psi_0$ a.e. in
$\mathbb{R}^2$}.\ee Furthermore,
\be\label{b14}\p\{\psi_n>0\}\rightarrow\p\{\psi_0>0\}\ \text{locally
in the Hausdorff metric}.\ee In particular, $0$ is the free boundary
point of $\psi_0$.

In view of \eqref{b12} and \eqref{b13}, there exist two subsequences
$\{\t Y_n\}$ and $\{\psi_n\}$, such that $\t Y_n\rightarrow \t Y_0$
and \be\label{b15}\fint_{\p B_{\f12}(\t Y_0)}\psi_0^+(\t X)dS_{\t
X}\geq c|\ld_L|^{\f12}\ \ \text{and}\ \ |\t Y_0|=1.\ee

 By using the monotonicity
formula lemma 5.1 in \cite{ACF4}, one has
$$\f1{r^4}\int_{B_r(\t Y_0)}|\g\psi_0^+(\t X)|^2d\t X\cdot\int_{B_r(\t Y_0)}|\g\psi_0^-(\t X)|^2d\t X=\gamma\geq
0,$$ for any $r>0$, which together with Lemma 6.6 in \cite{ACF4}
gives that \be\label{b16}\text{$\psi_0$ is either a 2-plane
solution, or a 1-plane solution, or identically zero.}\ee

Since $\th>0$, one has \be\label{b17}\psi_{\ld_L,L}\equiv 0\ \
\text{in $S$ and}\ \f{\mathcal{L}^2(B_r(A)\cap
S)}{\mathcal{L}^2(B_r(A))}=\f{\th}{2\pi}>0,\ee for any $r>0$, where
$S=\{(x,y)\mid x\leq y\cot\th, y\leq 0\}$. By virtue of \eqref{b16}
and \eqref{b17}, one has
$$\text{$\psi_0$ is either a 1-plane solution, or identically
zero,}$$ which together with \eqref{b15} gives that
\be\label{b18}\psi_0\geq 0\ \ \text{and}\ \ \ \psi_0\not\equiv 0.\ee

Since $\psi_{\ld_L,L}(X)\leq 0$ in $\{(x,y)\mid y\leq 0\}$ and
$\psi_{\ld_L,L}(X)>0$ in $\{(x,y)\mid x<0, y>0\}$, by virtue of
\eqref{b13} and \eqref{b18}, we have \be\label{b19}\psi_0(\t x,\t
y)\equiv 0\ \ \text{in $\{(\t x,\t y)\mid \t y\leq 0\}$},\ee and
\be\label{b20}\psi_0(\t x,\t y)> 0\ \ \text{in $\{(\t x,\t y)\mid \t
x<0, \t y>0\}$}.\ee

Thus, it follows from \eqref{b20} that \be\label{b21}\psi_0(\t x,\t
y)=\max\{\alpha\t y,0\}\ \ \text{and}\ \ \alpha>0.\ee

By virtue of the proof of Lemma 6.2 in \cite{ACF4}, we can conclude
that $\psi_0$ is a local minimizer for the variational problem
$$J_R(\psi_0)=\min J_R(\psi)\ \ \text{for any $\psi-\psi_0\in
H_0^1(B_R)$ and $R>0$},$$ where $B_R=B_R(0)$ and the functional
$$J_R(\psi)=\int_{B_R}\left|\nabla\psi-(\ld_{1,L}I_{\{\psi\leq0\}}+\ld_{2,L}I_{\{\psi>0\}})I_{\{\t
x>0\}}e\right|^2d\t xd\t y,$$ with $\ld_{1,L}=\f{Q}{h_L-b}$ and $
\ld_{2,L}=\f{L}{L-h_L}$. Next, we claim that $\psi_0$ is a local
minimizer for the variational problem, namely,
$$J_R^0(\psi_0)=\min J_R^0(\psi)\ \ \text{for any $\psi-\psi_0\in
H_0^1(B_R)$ and $R>0$},$$ where the functional
$$J^0_R(\psi)=\int_{B_R}|\nabla\psi|^2+(\ld_{2,L}^2-\ld_{1,L}^2)I_{\{\psi>0\}}I_{\{\t
x>0\}}d\t xd\t y.$$ In fact, for any $\psi-\psi_0\in H_0^1(B_R)$,
one has
$$\ba{rl}&\int_{B_R}|\nabla\psi_0|^2-\ld_L^2I_{\{\psi_0>0\}}I_{\{\t
x>0\}}d\t xd\t
y-\int_{B_R}|\nabla\psi|^2-\ld_L^2I_{\{\psi>0\}}I_{\{\t
x>0\}}d\t xd\t y\\
=&\int_{B_R}|\nabla\psi_0|^2+(\ld^2_{1,L}I_{\{\psi_0\leq0\}}+\ld^2_{2,L}I_{\{\psi_0>0\}})I_{\{\t
x>0\}} d\t xd\t y\\
&-\int_{B_R}|\nabla\psi|^2+(\ld^2_{1,L}I_{\{\psi\leq0\}}+\ld^2_{2,L}I_{\{\psi>0\}})I_{\{\t
x>0\}} d\t xd\t y\\
=&\int_{B_R}\left|\nabla\psi_0-(\ld_{1,L}I_{\{\psi_0\leq0\}}+\ld_{2,L}I_{\{\psi_0>0\}})I_{\{\t
x>0\}}e\right|^2d\t xd\t y\\
&-\int_{B_R}\left|\nabla\psi-(\ld_{1,L}I_{\{\psi\leq0\}}+\ld_{2,L}I_{\{\psi>0\}})I_{\{\t
x>0\}}e\right|^2d\t xd\t y\\
\leq&0,\ea$$ where we have used the fact
$$\int_{B_R}\nabla\psi_0\cdot eI_{\{\psi_0\leq0\}}I_{\{\t
x>0\}} d\t xd\t y=\int_{B_R}\nabla\psi\cdot
eI_{\{\psi\leq0\}}I_{\{\t x>0\}} d\t xd\t y,$$ and $$
\int_{B_R}\nabla\psi_0\cdot eI_{\{\psi_0>0\}}I_{\{\t x>0\}} d\t xd\t
y=\int_{B_R}\nabla\psi\cdot eI_{\{\psi>0\}}I_{\{\t x>0\}} d\t xd\t
y.$$

For the minimal functional $J^0_R(\psi_0)$, it follows from Theorem
2.5 in \cite{AC1} that
$$|\g\psi_0|^2=\ld_{2,L}^2-\ld_{1,L}^2=-\ld_L\ \ \text{on the free boundary of
$\psi_0$},$$ which implies that
$$\alpha=\sqrt{-\ld_L}.$$
 Hence, one has
 $$\psi_n(\t x, \t y)\rightarrow
\sqrt{-\ld_L}\t y^+\ \ \text{as $r_n\rightarrow0$.}$$

By using the similar arguments in Lemma 11.2 in Chapter 3 in
\cite{FA1}, we can conclude that $k'_{\ld_L,L}(0+0)=0$.

{\bf Case 2}. $\ld_L>0$. Similar to Case 1, we can conclude that
$k'_{\ld_L,L}(0+0)=\tan\th$.

{\bf Case 3}. $\ld_L=0$. It is easy to check that $\psi_{\ld_L,L}$
is a harmonic function across the free boundary in $\O_L$. By using
a conformal mapping
$\t\psi(z)=\psi_{\ld_L,L}\left(z^{\f{\pi}{2\pi-\th}}\right)$ with
$z=x+iy$, such that $\t\psi(z)$ becomes a harmonic function in
$B_r(0)\cap\{\text{Im}z>0\}$, $\t\psi=0$ on
$B_r(0)\cap\p\{\text{Im}z>0\}$, where $z=x+iy$. Furthermore,
$N_1\cup S_1$ is mapped into the real axis and the free boundary
$\Gamma_{\ld_L,L}$ is mapped into a continuous arc $\gamma$
initiating at $0$. Then the harmonic function $\t\psi(z)$ has
harmonic continuation across Im$z=0$ in $B_r(0)$. It follows that
the level set $\{\t\psi=0\}$ consists arcs forming equal angles at
$A$. Since $\t\psi$ vanishes only on $\gamma$, which implies that
the continuous arc $\gamma$ must intersect $\p\{\text{Im}z>0\}$
orthogonally at $A$, namely,  $k_{\ld_L,L}'(0+0)=\tan\f{\th}2$.

Hence, the condition (5) in Definition \ref{def2} is obtained.

{\bf Step 5.} In this step, we will verify that $\psi_{\ld_L,L}$
satisfies the condition (6) in Definition \ref{def2}. The asymptotic
behavior of $\psi_{\ld_L,L}$ in downstream has been obtained in Step
2. Next, we consider the asymptotic behavior of $\psi_{\ld_L,L}$ in
upstream. Define a blow-up sequence
$\psi_n(x,y)=\psi_{\ld_L,L}(x-n,y)$ for $x<\f n2$. By using the
elliptic regularity in \cite{GT}, there exists a subsequence
$\{\psi_n\}$, such that
$$\psi_n(x,y)\rightarrow\bar\psi_0(x,y) \ \ \text{in}\ \ C^{2,\alpha}(D),$$ for any compact subset $D$ of $E=\{-\infty<x<+\infty\}\times\{0<y<L\}$, and
$$\left\{\ba{ll} &\Delta\bar\psi_0=0  \ \text{in}~~ E,\\
&\bar\psi_0(x,0)=0\ \text{and}\ \bar\psi_0(x,L)=L\ \ \text{for
$-\infty<x<+\infty$},\\
&0\leq\bar\psi_0(x,y)\leq L\ \ \text{in}~~ E. \ea\right.
$$ Then the above boundary value problem has a unique solution
$$\bar\psi_0(x,y)=y\ \ \text{in}\ \
E,$$ which implies that \be\label{b22}\psi_{\ld_L,L}(x,y)\ra y\ \
\text{for}\ \ 0<y<L,\ \ \text{as} \ \ x\ra-\infty.\ee

Denote $\t x=x\sin\th-y\cos\th$ and $\t y=y\sin\th+x\cos\th$, let
$$\t\psi(\t x, \t y)=\psi_{\ld_L,L}(\t x\sin\theta+\t y\cos\theta, \t y\sin \theta - \t x\cos \theta)$$ and $\t\psi_n(\t x, \t
y)=\t\psi(\t x,\t y-n)$. By virtue of the elliptic regularity, there
exists a subsequence $\{\t\psi_n\}$, such that
$$\t\psi_n(\t x,\t y)\rightarrow\t\psi_0(\t x,\t y)\ \ \text{in}\ \ C^{2,\alpha}(D),$$ for any compact subset $D$ of
$\t E=\{0<\t x<a\sin\th-b\cos\th\}\times\{-\infty<\t y<+\infty\}$,
and $\t\psi_0$ satisfies that $$\left\{\ba{ll} &\Delta\t\psi_0=0  \ \text{in}~~\t E,\\
&\t \psi_0(0,\t y)=0\ \text{and}\ \t\psi_0(a\sin\th-b\cos\th,\t
y)=-Q\ \ \text{for
$-\infty<\t y<+\infty$},\\
&-Q\leq\t\psi_0(\t x,\t y)\leq 0\ \ \text{in}~~\t E. \ea\right.
$$ Then one has
$$\t\psi_0(\t x,\t y)=-\f{Q}{a\sin\th-b\cos\th}\t x\ \ \text{in}\ \
\t E,$$ which implies that
\be\label{b220}\left|\psi_{\ld_L,L}(x,y)-\f{Q(y\cos\th-x\sin\th)}{a\sin\th-b\cos\th}\right|\ra
0 \ \ \text{uniformly in any compact subset of $S$,}\ee as
$y\ra-\infty$, where $S=\{(x,y)\mid
y\cot\th<x<(y-b)\cot\th+a,-\infty<y<+\infty\}$.

{\bf Step 6.} Finally, we will verify the condition (7) in
Definition \ref{def2} and complete the proof. For any $\e>0$, by
virtue of \eqref{b10} and \eqref{b22}, there exists a large
$\mu_0>0$, such that
$$\psi_{\ld_L,L}-y\leq \e\ \ \text{in $\O_L^+\cap\{x\leq-\mu_0\}$
and $\psi_{\ld_L,L}\leq y\ $ in $\ \O_L^+\cap\{x\geq\mu_0\}$},
$$ where $\O_L^+=\O_L\cap\{\psi_{\ld_L,L}>0\}$. This together with the maximum principle gives that
$$\psi_{\ld_L,L}-y\leq \e\ \  \text{in $\O_L^+\cap\{-\mu_0\leq x\leq \mu_0\}$}.$$ Therefore, we
have \be\label{b23}\psi_{\ld_L,L}-y\leq \e\ \  \text{in $\
\O_L^+$}.\ee Taking $\e\ra 0$ in \eqref{b23}, one has
$$\text{$\psi_{\ld_L,L}^+(x,y)\leq y$ in
$\O_L\cap\{y>0\}$.}$$ Similarly, we can show that
$$\text{$\psi^+_{\ld_L,L}(x,y)\geq\f{L(y-h_L)}{L-h_L}\ $ in
$\ \O_L\cap\{y>0\}$.}$$

\end{proof}

Finally, we will obtain the positivity of horizontal velocity and
vertical velocity in the following.
\begin{lemma}\label{lb7} The horizontal velocity and vertical velocity are positive in $\O_L$, namely,
$$u>0\ \text{and} \ v>0\ \ \text{in}\ \ \ \O_{L}^-\cup\Gamma_{\ld_L,L}, \ \ \ u>0\ \ \text{and} \ v>0 \ \text{in}\ \ \ \O_{L}^+\cup\Gamma_{\ld_L,L},$$
where $\O_{L}^-=\O_L\cap\{\psi_{\ld_L,L}<0\}$ and
$\O_{L}^+=\O_L\cap\{\psi_{\ld_L,L}>0\}$.
\end{lemma}

\begin{proof}

Denote $\o_1(x,y)=\p_y\psi_{\ld_L,L}(x,y)$ in $\O_{L}^-$ and
$\o_2(x,y)=\p_y\psi_{\ld_L,L}(x,y)$ in $\O_{L}^+$, it is easy to
check that
$$\Delta\o_1=0\ \ \text{in}\ \ \O_{L}^-,\ \ \text{and}\ \ \Delta\o_2=0\ \ \text{in}\ \ \O_{L}^+.$$
Since $\psi_{\ld_L,L}(x,y)$ is monotone increasing with respect to
$y$, which together with the strong maximum principle gives that
$$\text{$\o_1(x,y)>0$ in $\O_{L}^-$, and $\o_2(x,y)>0$ in $\O_{L}^+$}.$$

Next, we claim that
$$\o_1(x,y)>0\ \ \text{and}\ \ \o_2(x,y)>0\ \ \text{on}\ \ \Gamma_{\ld_L,L}.$$

Suppose not, without loss of generality, we assume that there exists
an $x_0\in(0,+\infty)$, such that $\o_1(X_0)=0$ with
$X_0=(x_0,k_{\ld_L,L}(x_0))$. We consider the following two cases.

{\bf Case 1.} $\ld_L=0$. Then we have that $\psi_{\ld_L,L}$ is
harmonic in $\O_L$, the strong maximum principle gives that
$\o_1(X_0)>0$, which contradicts to our assumption.

{\bf Case 2.} $\ld_L\neq 0$. Since the free boundary
$\Gamma_{\ld_L,L}$ is analytic at $X_0$, $\o_1(X_0)=0$ implies that
the normal vector of $\Gamma_{\ld_L,L}$ is parallel to $(1,0)$,
which implies that $\o_2(X_0)$ is also zero. Thus, it follows from
\eqref{b3} that
$|\p_x\psi_{\ld_L,L}^-(X_0)|^2-|\p_x\psi_{\ld_L,L}^+(X_0)|^2=\ld_L$.

Without loss of generality, we assume that the outer normal vector
of $\p\{\psi_{\ld_L,L}>0\}$ at $X_0$ is $\nu=(1,0)$. Thanks to
Hopf's lemma, one has
\be\label{b24}\p_{x}\psi_{\ld_L,L}^+=\f{\p\psi_{\ld_L,L}^+}{\p\nu}<0\
\ \text{and}\ \
\p_{x}\psi_{\ld_L,L}^-=\f{\p\psi_{\ld_L,L}^-}{\p\nu}>0\ \ \text{on}\
\ X_0,\ee and
\be\label{b25}\p^2_{xy}\psi_{\ld_L,L}^+=\p_x\o_1=\f{\p\o_1}{\p\nu}<0\
\ \text{and}\ \
\p^2_{xy}\psi_{\ld_L,L}^-=\p_x\o_2=\f{\p\o_2}{\p\nu}<0\ \ \text{on}\
\ X_0.\ee

Since $|\g\psi_{\ld_L,L}^-|^2-|\g\psi_{\ld_L,L}^+|^2=\ld_L$ on the
free boundary $\Gamma_{\ld_L,L}$, one has
\be\label{b26}0=\f{\p(|\g\psi_{\ld_L,L}^-|^2-|\g\psi_{\ld_L,L}^+|^2)}{\p
s}=2(\p_{xy}\psi_{\ld_L,L}^-\p_{x}\psi_{\ld_L,L}^-
-\p_{xy}\psi_{\ld_L,L}^+\p_{x}\psi_{\ld_L,L}^+)\ \ \text{at}\ \
X_0,\ee where $s=(0,1)$ is the tangential direction of
$\Gamma_{\ld_L,L}$ at $X_0$. On the other hand,  it follows from
\eqref{b24} and \eqref{b25} that
$$\p^2_{xy}\psi_{\ld_L,L}^-\p_{x}\psi_{\ld_L,L}^-
-\p^2_{xy}\psi_{\ld_L,L}^+\p_{x}\psi_{\ld_L,L}^+<0\ \ \text{at}\ \
X_0,$$ which contradicts to \eqref{b26}.

Similarly, we can show that $$v>0\ \ \text{in}\ \ \
\O_{L}^-\cup\Gamma_{\ld_L,L}\ \ \text{and}\ \ v>0 \ \text{in}\ \ \
\O_{L}^+\cup\Gamma_{\ld_L,L}.$$

\end{proof}

\subsection{The uniqueness of the truncated injection flow problem 1}
In this subsection, we will obtain the uniqueness of the truncated
injection flow problem 1 for any given $Q>0$ and $L>b$.

\begin{lemma}\label{lb8}For any $Q>0$ and $L>b$, there exist a unique $\ld_L$ and a unique solution
$(\psi_{\ld_L,L},\Gamma_{\ld_L,L})$ to the truncated injection flow
problem 1.

\end{lemma}
\begin{proof}

Suppose that there exist two different solutions
$\psi=\psi_{\ld_L,L}$ and $\t\psi=\t\psi_{\t\ld_L,L}$. We divide two
steps to complete the proof.

{\bf Step 1.} First, we show that
$$\ld_L=\t\ld_L.$$
Suppose not, without loss of generality, we assume that
$\ld_L<\t\ld_L$. Noting that
$$\ld_L=\f{Q^2}{(h_L-b)^2}-\f{L^2}{(L-h_L)^2}\ \ \text{and}\ \ \t\ld_L=\f{Q^2}{(\t h_L-b)^2}-\f{L^2}{(L-\t h_L)^2},$$ this together with Remark \ref{re2} implies that
\be\label{b27}\lim_{x\ra+\infty}k_{\ld_L,L}(x)=h_L>\t
h_L=\lim_{x\ra+\infty}\t k_{\t\ld_L,L}(x).\ee Then,
\be\label{b28}k_{\ld_L,L}(x)>\t k_{\t\ld_L,L}(x)\ \ \text{for
sufficiently large $x>0$}.\ee

Define a function $\psi_\e(x,y)=\psi(x,y-\e)$ for any $\e\geq0$, and
$\Gamma_{\ld_L,L}^\e:y=k_{\ld_L,L}(x)+\e$ as the free boundary of
$\psi_\e$. Take $\e_0\geq0$ to be the smallest one such that
\be\label{b29}\psi_{\e_0}(X)\leq\t\psi(X)\ \ \text{in $\O_L$ and
$\psi_{\e_0}(X_0)=\t\psi(X_0)$ for some $X_0\in\bar\O_L$.}\ee

We consider the following two cases for $\e_0$.

{\bf Case 1.} $\e_0>0$. The strong maximum principle gives that
$X_0\notin\O_L\cap(\{\t\psi<0\}\cup\{\psi_{\e_0}>0\})$. Suppose not,
without loss of generality, we assume that there exists
$X_0\in\O_L\cap\{\t\psi<0\}$, such that
$-Q<\psi_{\e_0}(X_0)=\t\psi(X_0)<0$. The continuity of $\psi_{\e_0}$
and $\t\psi$ implies that there exists a small $r>0$, such that
$$-Q<\psi_{\e_0}(X)<0\ \ \text{and}\ \ -Q<\t\psi(X)<0\ \ \text{in
$B_r(X_0)$}.$$ Since $\psi_{\e_0}$ and $\t\psi$ are harmonic in
$B_r(X_0)$, the strong maximum principle gives that
$\psi_{\e_0}\equiv\t\psi$ in $B_r(X_0)$, due to
$\psi_{\e_0}(X_0)=\t\psi(X_0)$. Applying the strong maximum
principle again, we can obtain a contradiction to the boundary value
of $\t\psi$.

Since $\e_0>0$, it follows from \eqref{b28} that $|X_0|<+\infty$.
Therefore, choose $X_0$ to be a free boundary point of $\psi_{\e_0}$
and $\t\psi$, and one has $\psi_{\e_0}(X_0)=\t\psi(X_0)=0$. In view
of \eqref{b29}, the strong maximum principle gives that
$$\psi_{\e_0}<\t\psi\ \ \text{in}\ \ \O_L\cap\{\t\psi<0\}\ \ \text{and}\ \ \psi_{\e_0}<\t\psi\ \ \text{in}\ \ \O_L\cap\{\psi_{\e_0}>0\}.$$
Since the free boundaries $\Gamma_{\ld_L,L}^{\e_0}$ and
$\t\Gamma_{\t\ld_L,L}$ are analytic at $X_0$, thanks to Hopf's
lemma, one has
$$|\g\psi_{\e_0}^-|=-\f{\p\psi^-_{\e_0}}{\p\nu}>-\f{\p\t\psi^-}{\p\nu}=|\g\t\psi^-|\ \ \text{and}\ \
|\g\psi_{\e_0}^+|=\f{\p\psi^+_{\e_0}}{\p\nu}<\f{\p\t\psi^+}{\p\nu}=|\g\t\psi^+|\
\ \text{at}\ \ X_0,$$ where $\nu$ is the inner normal vector to
$\p\{\t\psi>0\}$ at $X_0$. Those give that
$$\ld_L=|\g\psi_{\e_0}^-|^2-|\g\psi_{\e_0}^+|^2>|\g\t\psi^-|^2-|\g\t\psi^+|^2=\t\ld_L \ \ \text{at}\ \ X_0,$$
which contradicts to our assumption.

{\bf Case 2.} $\e_0=0$. We first claim that
\be\label{b30}\ld_L\cdot\t\ld_L>0.\ee Suppose not, if
$\t\ld_L>0\geq\ld_L$, by virtue of \eqref{b2}, one has
$$\t k_{\ld_L,L}'(0+0)=\tan\th>\tan\f\th2\geq k'_{\t\ld_L,L}(0+0),$$ which implies that $k_{\ld_L,L}(x)<\t
k_{\t\ld_L,L}(x)$ for small $x>0$. This leads a contradiction to the
fact $\psi\geq\t\psi$ in $\O$. Similarly, we can obtain a
contradiction if $\t\ld_L=0>\ld_L$.

In view of the claim \eqref{b30}, without loss of generality, we
assume that $0>\t\ld_L>\ld_L$ and take $X_0=A$. Define two blow-up
sequences $\{\psi_n\}$ and $\{\t\psi_n\}$ with $\psi_n(\t
X)=\f{\psi(r_n\t X)}{r_n}$ and $\t\psi_n(\t X)=\f{\t\psi(r_n\t
X)}{r_n}$. Let $\psi_0$ and $\t\psi_0$ be the blow-up limits of
$\psi_n$ and $\t\psi_n$ as $r_n\ra0$, respectively.

It follows from the similar arguments in Step 3 in the proof of
Theorem \ref{lb6} that $\psi_0$ and $\t\psi_0$ satisfy that
\be\label{b31}\psi_0(\t x,\t y)=\max\{\alpha \t y,0\},\ \ \
\alpha>0\ \ \text{and}\ \ \alpha^2=-\ld_L, \ee and
\be\label{b32}\t\psi_0(\t x,\t y)=\max\{\t\alpha \t y,0\},\ \
\t\alpha>0\ \ \text{and}\ \ \t\alpha^2=-\t\ld_L. \ee

The fact \eqref{b29} implies that
$$\t\psi_0\geq\psi_0\ \ \text{in
$B_1(0)\cap\{\t\psi_0>0\}$},$$ which together with \eqref{b31} and
\eqref{b32} implies that
$$\sqrt{-\t\ld_L}=\f{\p\t\psi_0}{\p\nu}\geq\f{\p\psi_0}{\p\nu}=\sqrt{-\ld_L}\ \ \text{at}\ \
0,$$ where $\nu=(0,1)$ is inner normal vector. This contradicts to
our assumption $0>\t\ld_L>\ld_L$.

{\bf Step 2.} In this step, we will show that $\psi=\t\psi$. It
follows from the asymptotic behavior of $\psi$ and $\t\psi$ that
$$\lim_{x\rightarrow+\infty}k_{\ld_L,L}(x)=\lim_{x\rightarrow+\infty}\t k_{\ld_L,L}(x).$$
Without loss of generality, we assume that there exists $x_0>0$,
such that \be\label{b33}k_{\ld_L,L}(x_0)<\t k_{\ld_L,L}(x_0).\ee

Consider a function $\psi_{\e}(x,y)=\psi(x,y-\e)$ for $\e\geq0$, and
choosing the smallest $\e_0\geq0$ such that
$$\psi_{\e_0}(X)\leq\t\psi(X)\ \text{in $\O$, and}\ \psi_{\e_0}(X_0)=\t\psi(X_0)\ \text{for some $X_0\in\bar{\O}$}.$$
It follows from \eqref{b33} that $\e_0>0$, which implies that
$|X_0|<+\infty$. By using the similar arguments in Step 1,  we can
let $X_0$ be the free boundary point of $\psi_{\e_0}$ and $\t\psi$.
Applying Hopf's lemma at $X_0$, one has
$$\ld_L=|\g\psi_{\e_0}^-|^2-|\g\psi_{\e_0}^+|^2>|\g\t\psi^-|^2-|\g\t\psi^+|^2=\ld_L\ \ \text{at}\ \ X_0,$$ which is impossible.

Hence, we obtain the uniqueness of the solution to the truncated
injection flow problem 1 for any $L>b$.
\end{proof}

\subsection{The relation between $\ld_L$ and $Q$} In Lemma
\ref{lb8}, the uniqueness of $\ld_L$ is obtained for any $Q>0$ and
$L>b$. Then we can denote $\ld_L=\ld_L(Q)$ for any $Q>0$ and $L>b$.
We investigate the relation between $\ld_L(Q)$ and $Q$ for any fixed
$L>b$, and show that $\ld_L(Q)$ is strictly monotone increasing and
continuous with respect to $Q>0$ for any $L>b$.

\begin{lemma}\label{lb9}For any $L>b$, $\ld_L(Q)$ is strictly monotone increasing and continuous with respect to $Q$.

\end{lemma}
\begin{proof}
For any $Q_1>Q_2>0$, there exist a unique $\ld_L(Q_1)$ and a unique
$\ld_L(Q_2)$, such that
$(\psi_{\ld_L(Q_1),L},\Gamma_{\ld_L(Q_1),L})$ and
$(\psi_{\ld_L(Q_2),L},\Gamma_{\ld_L(Q_1),L})$ are the solutions to
the truncated injection flow problem 1. We next show that
$$\ld_L(Q_1)>\ld_L(Q_2)\ \ \text{for any $Q_1>Q_2>0$}.$$
 If not, suppose that there exist $Q_1>Q_2>0$, such that
 $\ld_L(Q_1)\leq\ld_L(Q_2)$. Denote $\psi_1=\psi_{\ld_L(Q_1),L}$ and $\psi_2=\psi_{\ld_L(Q_2),L}$ for simplicity. We next consider the following two
 cases.

 {\bf Case 1.} $\ld_L(Q_1)=\ld_L(Q_2)=0$. Then $\psi_1$ and $\psi_2$ are harmonic functions in $\O_L$.
For the harmonic function $\psi_1$ in $\O_L$, along the conformal
mapping, it follows from (7) in pp.292 in \cite{MI} that
 $$\left(\f{U_1}{U}\right)^{1-\f\pi\th}\f{d_1}{d}-\left(\f{U_2}{U}\right)^{1-\f\pi\th}\f{d_2}{d}=1,$$
 where $d=L-b$, $d_1=L$, $d_2=a\sin\th-b\cos\th$, $U=\f{L+Q_1}{d}$,
 $U_1=\f{L}{d_1}$ and $U_2=\f{Q_1}{d_2}$. Then we have
\be\label{b34}d_2^{\f\pi\th}Q_1^{1-\f\pi\th}=L-(L-b)^{\f\pi\th}(L+Q_1)^{1-\f\pi\th}.\ee
Define
$f(Q)=d_2^{\f\pi\th}Q^{1-\f\pi\th}-L+(L-b)^{\f\pi\th}(L+Q)^{1-\f\pi\th}$
for $Q>0$, it is easy to check that $f(Q)$ is strictly monotone
decreasing with respect to $Q>0$, which implies that $Q_1>0$ is
uniquely determined by \eqref{b34}. Similarly, for the harmonic
function $\psi_2$ in $\O_L$, one has
$$d_2^{\f\pi\th}Q_2^{1-\f\pi\th}=L-(L-b)^{\f\pi\th}(L+Q_2)^{1-\f\pi\th},$$
which implies that $Q_1=Q_2$. This leads a contradiction.

  {\bf Case 2.} $\ld_L(Q_1)\neq0$ or $\ld_L(Q_2)\neq0$.

Recalling Remark \ref{re2}, there exists a unique $h_{i,L}$
($i=1,2$), such that
$$\ld_L(Q_1)=\f{Q_1^2}{(h_{1,L}-b)^2}-\f{L^2}{(L-h_{1,L})^2}\ \ \text{and}\ \ \ld_L(Q_2)=\f{Q_2^2}{(h_{2,L}-b)^2}-\f{L^2}{(L-h_{2,L})^2}.$$
Moreover, $Q_1>Q_2$ implies that
$$h_{1,L}>h_{2,L}$$ and furthermore, $$k_{\ld_L(Q_1),L}(x)>k_{\ld_L(Q_2),L}(x)\ \ \text{for sufficiently large $x>0$}.$$

Define a function $\psi_{1,\e}(x,y)=\psi_1(x,y-\e)$ for $\e\geq0$,
and let $\e_0\geq0$ be the smallest one such that
$$\psi_{1,\e_0}(X)\leq\psi_2(X)\ \ \text{in $\O_L$ and $\psi_{1,\e_0}(X_0)=\psi_2(X_0)$ for some $X_0\in\bar\O_L$.}$$

We consider the following two subcases.

{\bf Subcase 2.1.} $\e_0>0$. Similar to Case 1 in the proof of Lemma
\ref{lb8}, we can conclude that
$X_0\notin\O_L\cap(\{\psi_2<0\}\cup\{\psi_{1,\e_0}>0\})$ and
$|X_0|<+\infty$. Therefore, choose $X_0$ be a free boundary point of
$\psi_{1,\e_0}$ and $\psi_2$. Thanks to Hopf's lemma, one has
$$\text{$\ld_L(Q_1)=|\g\psi_{1,\e_0}^-|^2-|\g\psi_{1,\e_0}^+|^2>|\g\psi_2^-|^2-|\g\psi_2^+|^2=\ld_L(Q_2)$
at $X_0$,}$$ which contradicts to our assumption
$\ld_L(Q_1)\leq\ld_L(Q_2)$.

{\bf Subcase 2.2.} $\e_0=0$. Along the proof of the claim
\eqref{b30}, we have that $\ld_L(Q_1)\cdot\ld_L(Q_2)>0$. Without
loss of generality, we assume that $0>\ld_L(Q_2)\geq\ld_L(Q_1)$.
Taking $X_0=A$, the strong maximum principle gives that
$$\psi_1<\psi_2\ \ \text{in}\ \
\O\cap\{\psi_1>0\}.$$ We next show that
\be\label{b35}k_{\ld_L(Q_1),L}(x)>k_{\ld_L(Q_2),L}(x)\ \ \text{for
any $x>0$}.\ee If not, there exists an $x_1\in(0,+\infty)$, such
that $k_{\ld_L(Q_1),L}(x_1)=\t k_{\ld_L(Q_2),L}(x_1)$. Taking
$X_0=(x_1,k_{\ld_L(Q_1),L})$ as the free boundary point, we can
obtain a contradiction by using the similar arguments in Subcase
2.1,

Since $N_1$ is $C^{2,\alpha}$-smooth, by using Hopf's lemma, one has
\be\label{b36}\f{\p\psi_1}{\p\nu}<\f{\p\psi_2}{\p\nu}\ \ \text{on
$N_{1}\cap\{x<0\}$, $\nu$ is the inner normal vector of $N_{1}$}.\ee

In view of \eqref{b35} and \eqref{b36}, for small $r>0$, there
exists a small $\delta>0$, such that
$$
(1+\delta)\psi_1\leq\psi_2\ \ \text{on
$\p(B_r(0)\cap\{\psi_1>0\})$}.$$ The maximum principle gives that
\be\label{b37}(1+\delta)\psi_1\leq\psi_2\ \ \text{in
$B_r(0)\cap\{\psi_1>0\}$}.\ee Define two blow-up sequences
$\{\psi_{1,n}\}$ and $\{\psi_{2,n}\}$ with $\psi_{1,n}(\t
X)=\f{\psi_1(r_n\t X)}{r_n}$ and $\psi_{2,n}(\t X)=\f{\psi_2(r_n\t
X)}{r_n}$. Since $\psi_{1,n}$ and $\t\psi_{1,n}$ are Lipschitz
continuous, we can denote $\psi_{1,0}$ and $\psi_{2,0}$ as the
blow-up limit of $\psi_{1,n}$ and $\psi_{2,n}$, respectively.
Furthermore,
$$\psi_{1,0}(\t x,\t y)=\max\{\sqrt{-\ld_L(Q_1)}\t y,0\}\ \ \
\text{and}\ \ \psi_{2,0}(\t x,\t y)=\max\{\sqrt{-\ld_L(Q_2)} \t
y,0\}, $$ and
$$\psi_{2,0}\geq(1+\delta)\psi_{1,0}\ \ \text{in
$\{\t y>0\}$}.$$ This give that
$$\sqrt{-\t\ld_L(Q_2)}=\f{\p\psi_{2,0}}{\p\nu}\geq(1+\delta)\f{\p\psi_{1,0}}{\p\nu}=(1+\delta)\sqrt{-\ld_L(Q_1)}\ \ \text{at}\ \
0,$$ where $\nu=(0,1)$ is the inner normal vector. This contradicts
to our assumption $0>\ld_L(Q_2)\geq\ld_L(Q_1)$.

Next, we will show that $\ld_L(Q)$ is continuous with respect to
$Q$. Since $\ld_L(Q)$ is strictly monotone increasing with respect
to $Q$, it suffices to show that $$\text{$\ld_L(Q+0)=\ld_L(Q-0)$ for
any $Q>0$,}$$ where $\ld_L(Q+0)=\lim_{Q_n\rightarrow Q^+}\ld_L(Q_n)$
and $\ld_L(Q-0)=\lim_{Q_n\rightarrow Q^-}\ld_L(Q_n)$.

Suppose not, then there exists a $Q_0>0$, such that
$\ld_L(Q_0+0)>\ld_L(Q_0-0)$. For a sequence $\{Q_n\}$ with
$Q_n\downarrow Q_0$, there exist a unique $\ld_L(Q_n)$ and a unique
solution $\psi_{\ld_L(Q_n),L}$ to the truncated injection flow
problem 1. Then there exists a subsequence $\{Q_n\}$, such that
$$\ld_L(Q_n)\rightarrow\ld_L(Q_0+0),$$ and $$\psi_{\ld_L(Q_n),L}\rightarrow\psi_{\ld_L(Q_0+0),L}\ \text{in $H^1_{loc}(\O_L)$
and uniformly in any compact subset of $\O_L$}.$$ It is easy to
check that $(\psi_{\ld_L(Q_0+0),L},\Gamma_{\ld_L(Q_0+0),L})$ is a
solution to the truncated injection flow problem 1.

Similarly, there exists a solution
$(\psi_{\ld_L(Q_0-0),L},\Gamma_{\ld_L(Q_0-0),L})$ to the truncated
injection flow problem 1.

For the given $Q_0>0$, the uniqueness of $\ld_L$ and
$\psi_{\ld_L,L}$ gives that $\ld_L(Q_0+0)=\ld_L(Q_0-0)$ and
$\psi_{\ld_L(Q_0+0),L}=\psi_{\ld_L(Q_0-0),L}$, which leads a
contradiction.

\end{proof}

Next, we will obtain the upper bound and the lower bound of
$\ld_L(Q)$. It should be noted that the monotonicity of $\ld_L(Q)$
implies that the lower bound of $\ld_L(Q)$ follows from the limit
$\lim_{Q\rightarrow0}\ld_L(Q)$, which means the injection flow
vanishes. To see this, we have to investigate the one-phase flow
above a blade surface with unit velocity in upstream. In the special
case $b=0$ (see Figure \ref{f2}), the problem is so simple and the
one-phase flow for $Q=0$ is nothing but the uniform flow
$(u,v)=\f1{\sqrt{\rho_+}}(1,0)$ with free boundary
$\Gamma=\{0<x<a,y=0\}$. And it is clear that the limit
$\ld_L(Q)\rightarrow-1$ as $Q\rightarrow0$. However, at the present
situation ($b>0$), the one-phase fluid problem is unclear and
complicated, which is the one of main differences and difficulties
here. Therefore, there is an important observation that for the
limit case $Q=0$, the free boundary initiates at the leading edge
$A$ and touches the boundary $S_2$ below the trailing edge $B$ (see
Figure \ref{f11}). Based on this important observation, we will show
that the limit $\lim_{Q\rightarrow 0}\ld_L(Q)$ is not $-1$ but a
constant $\ld_L(0)>-1$. This is also a difference from the special
case $b=0$.

\begin{figure}[!h]
\includegraphics[width=100mm]{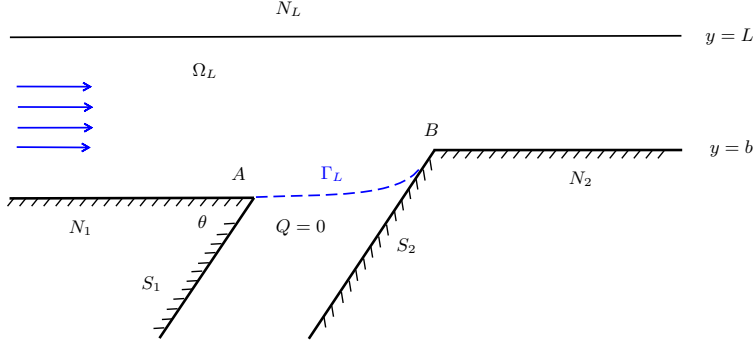}
\caption{The critical case $Q=0$}\label{f11}
\end{figure}

\begin{lemma}\label{lb10} For any $L>b$, there exist a $\ld_L(0)\in(-1,0)$ and a $\kappa_L\in(0,+\infty)$, such that $$\ld_L(Q)\rightarrow\ld_L(0)\ \ \text{as}\ \
Q\rightarrow0,$$ and $$\f{\ld_L(Q)}{Q^2}\rightarrow\kappa_L\ \
\text{as}\ \ Q\rightarrow+\infty.$$ Furthermore, if $L$ is
sufficiently large, $\kappa_L$ is a uniform constant independent of
$L$.

\end{lemma}
\begin{proof} {\bf Step 1.} The limit $Q\rightarrow0$. For any sequence $\{Q_n\}$ with $Q_n>0$ and $Q_n\rightarrow
0$, there exists a subsequence $\{Q_n\}$, such that
$$\ld_L(Q_n)\rightarrow\ld_L(0),$$ and
$$\psi_{\ld_L(Q_n),L}\rightarrow\psi_{\ld_L(0),L}\ \text{in
$H^1_{loc}(\O_L)$ and uniformly in any compact subset of $\O_L$},$$
as $Q_n\ra0$. Furthermore, $\psi_{\ld_L(0),L}(x,y)$ is monotone
increasing with respect to $y$ and decreasing with respect to $x$.
The monotonicity of $\psi_{\ld_L(0),L}(x,y)$ implies that there
exists a monotone increasing function $y=k_{\ld_L(0),L}(x)$ for
$x>0$, such that
$$\O_L\cap\{x>0\}\cap\{\psi_{\ld_L(0),L}>0\}=\O_L\cap\{x>0\}\cap\{y>k_{\ld_L(0),L}(x)\},$$
and $$k_{\ld_L(0),L}(0)=0.$$ We first claim that
\be\label{b38}k_{\ld_L(0),L}(x)\equiv b\ \ \text{for any
$x\in(a,+\infty)$}.\ee Suppose not, there exists an
$x_0\in(a,+\infty)$, such that $k_{\ld_L(0),L}(x_0)=b$ and
$k_{\ld_L(0),L}(x)>b$ for any $x\in(x_0,+\infty)$. By using the
asymptotic behavior of $\psi_{\ld_L(0),L}$ in the downstream, one
has \be\label{b380}\lim_{x\ra+\infty}k_{\ld_L(0),L}(x)=h_0\in(b,L)\
\ \text{and}\ \ -\ld_L(0)=\f{L^2}{(L-h_0)^2}.\ee It follows from the
results in Section 9 in \cite{ACF3} and Section 11 in Chapter 3 in
\cite{FA1} that the continuous fit condition implies the smooth fit
condition, namely, $N_1\cup\Gamma_{\ld_L(0),L}$ is $C^1$-smooth at
$A$ and $k'_{\ld_L(0),L}(0+0)=0$. Furthermore, $\g\psi_{\ld_L(0),L}$
is uniformly continuous in a $\{\psi_{\ld_L(0),L}>0\}$-neighborhood
of $A$.

Define $\o(y)=\max\{y,0\}$ for $y\in(-\infty,L)$, it is easy to
check that $\o(y)\geq\psi_{\ld_L(0),L}(x,y)$ in $\O_L$. In view of
$\o(0)=\psi_{\ld_L(0),L}(0,0)=0$, one has
$$1=\f{\p\o}{\p\nu}\geq \f{\psi_{\ld_L(0),L}}{\p\nu}=\sqrt{-\ld_L(0)}\ \ \text{at}\ \
A,$$ where $\nu=(0,1)$ is the inner normal vector. This contradicts
to the fact that $\ld_{L}(0)=-\f{L^2}{(L-h_0)^2}<-1$ in
\eqref{b380}.

Next, we will show that \be\label{b39}\ld_L(0)>-1.\ee If not, we
assume that $\ld_L(0)\leq-1$. For any small $r>0$, it follows from
the proof of \eqref{b37} that there exists a small $\delta>0$, such
that
$$\o\geq(1+\delta)\psi_{\ld_L(0),L}\ \ \text{in}\ \
B_r(A)\cap\{\psi_{\ld_L(0),L}>0\},$$ which gives that
$$1=\f{\p\o}{\p\nu}\geq (1+\delta)\f{\psi_{\ld_L(0),L}}{\p\nu}=(1+\delta)\sqrt{-\ld_L(0)}\ \ \text{at}\ \
A.$$ This contradicts to our assumption $\ld_{L}(0)\leq-1$.

Moreover, we will show that
\be\label{b40}\bar\Gamma_{\ld_L(0),L}\cap N_2=\varnothing.\ee If
not, it follows from \eqref{b38} that $k_{\ld_L(0),L}(a)=b$.
Similarly, we have that $N_2\cup\Gamma_{\ld_L(0),L}$ is $C^1$-smooth
at $B$ and $\g\psi_{\ld_L(0),L}$ is uniformly continuous in a
$\{\psi_{\ld_L(0),L}>0\}$-neighborhood of $B$. Define
$\o_1(y)=\f{L}{L-b}\max\{y-b,0\}$, it is easy to check that
$\o_1(y)\leq\psi_{\ld_L(0),L}(x,y)$ in $\O_L$, and thus
$$1<\f{L}{L-b}=\f{\p\o_1}{\p\nu}\leq \f{\psi_{\ld_L(0),L}}{\p\nu}=\sqrt{-\ld_L(0)}\ \ \text{at}\ \
B,$$ where $\nu=(0,1)$ is the inner normal vector. This contradicts
to \eqref{b39}.

Since $\ld_L(Q)$ is strictly decreasing with respect to $Q$, we can
obtain the uniqueness of $\ld_L(0)$.







{\bf Step 2.} The limit $Q\rightarrow+\infty$. We will show that
there exists a positive constant $\kappa_L$, such that
$$\f{\ld_L(Q)}{Q^2}\rightarrow\kappa_L\ \ \ \ \text{as}\ \
Q\rightarrow+\infty.$$

For any fixed $L>b$, set $\psi_Q=\f{\psi_{\ld_L(Q),L}}{Q}$ and
$\ld_Q=\f{\ld_L(Q)}{Q^2}$. Then $\psi_Q$ solves the following free
boundary value problem
 $$\left\{\ba{ll} &\Delta\psi_Q=0  \ \text{in}~~\O_L\cap\{\psi_Q<0\},\ \ \Delta\psi_Q=0  \ \text{in}~~\O_L\cap\{\psi_Q>0\},\\
&|\g\psi_Q^-|^2-|\g\psi_Q^+|^2=\ld_Q\ \ \text{on}~~\Gamma_{\ld_L(Q),L},\\
&\psi_Q=0\ \text{on}\  N_1 \cup S_1\cup\Gamma_{\ld_L(Q),L},\
\psi_Q=-1 \ \text{on} \ N_2\cup S_2,
 \ \ \psi_Q=\f{L}{Q}\
\text{on}\ N_L. \ea\right. $$

By virtue of non-degeneracy Theorem 3.1 in \cite{ACF4}, we have that
if $\ld_L(Q)>0$, then \be\label{b41}\f{1}{r}\fint_{\p B_r(X_0)}
\psi^-_{\ld_L(Q),L}dS\leq c\sqrt{\ld_L(Q)}\ \ \text{implies
$\psi_{\ld_L(Q),L}\equiv 0$ in $B_{\f r2}(X_0)$},\ee and if
$\ld_L(Q)<0$, then \be\label{b42}\f{1}{r}\fint_{\p B_r(X_0)}
\psi^+_{\ld_L(Q),L}dS\leq c\sqrt{-\ld_L(Q)}\ \ \text{implies
$\psi_{\ld_L(Q),L}\equiv 0$ in $B_{\f r2}(X_0)$},\ee for any disc
$B_r(X_0)\subset\O_L$ with $B_{\f r2}(X_0)\subset\O_L\cap\{(x,y)\mid
x>0,y>0\}$. Here, $c>0$ is a constant independent of $\ld_L(Q)$ and
$L$. Therefore, there exists a constant $r_0>0$ independent of
$\ld_L(Q)$ and $L$, such that $B_{r_0}(X_0)\subset\O_L$ with $B_{\f
{r_0}2}(X_0)\subset\O_L\cap\{(x,y)\mid x>0,y>0\}$ and
$B_{\f{r_0}2}(X_0)\cap\Gamma_{\ld_L(Q),L}\neq\varnothing$, and it
follows from \eqref{b41} and \eqref{b42} that
$$\f{Q}{r_0}\geq\f{1}{r_0}\left|\fint_{\p B_{r_0}(X_0)} \psi_{\ld_L(Q),L}dS\right|\geq
c|\ld_L(Q)|^{\f12},$$ for any $Q>0$. This implies that
\be\label{b43}|\ld_Q|\leq C,\ \ \text{$C>0$ is a constant
independent of $Q$ and $L$}.\ee

For any sequence $\{Q_n\}$ with $Q_n\ra+\infty$, there exists a
subsequence $\{Q_n\}$, such that
$$\ld_{Q_n}\rightarrow\kappa_L,$$ and
$$\psi_{Q_n}\rightarrow\bar\psi_{\kappa_L}\ \text{in
$H^1_{loc}(\O_L)$ and uniformly in any compact subset of $\O_L$},$$
as $Q_n\ra+\infty$. The monotonicity of $\psi_{\ld_L(Q),L}(x,y)$
with respect to $x$ and $y$ gives that $\bar\psi_{\kappa_L}(x,y)$ is
monotone increasing with respect to $y$ and decreasing with respect
to $x$.

Since $0\leq\psi_{Q}\leq\f{L}{Q}$ in $\O_L\cap\{(x,y)\mid x\leq
0,y\geq 0\}$, we have that $\bar\psi_{\kappa_L}=0$ in
$\O_L\cap\{(x,y)\mid x\leq 0,y\geq 0\}$. Denote
$E_L=\O_L\setminus(\{(x,y)\mid x\leq 0,y\geq 0\})$, then
$\bar\psi_{\kappa_L}$ is a solution of the following free boundary
value problem
\be\label{b44}\left\{\ba{ll} &\Delta\bar\psi_{\kappa_L}=0  \ \text{in}~~E_L\cap\{\bar\psi_{\kappa_L}<0\},\\
&\bar\psi_{\kappa_L}=0,\ \ \left|\f{\p\bar\psi_{\kappa_L}}{\p\nu}\right|^2=\kappa_L\ \ \text{on $\Gamma_{\kappa_L}$},\\
&\bar\psi_{\kappa_L}=0\ \text{on}\  N_L^+\cup I_L\cup S_1,\
\bar\psi_{\kappa_L}=-1 \ \text{on} \ N_2\cup S_2, \ea\right. \ee
where $N_L^+=N_L\cap\{x\geq 0\}$, $I_L=\{(0,y)\mid 0\leq y\leq L\}$
and $\Gamma_{\kappa_L}=E_L\cap\p\{\bar\psi_{\kappa_L}<0\}$ is the
free boundary of $\bar\psi_{\kappa_L}$. Furthermore, the free
boundary $\Gamma_{\kappa_L}$ is $C^1$-smooth at the initial point
$A$, and which is given by
$$\Gamma_{\kappa_L}=\{(x,y)\mid x=g_{\kappa_L}(y), 0<y<h_L\},\ \ \text{$g_{\kappa_L}(y)$ is increasing with respect to $y$,}$$ where either
$h_L<L$, $g_{\kappa_L}(h_L-0)=+\infty$ or $h_L=L$,
$g_{\kappa_L}(h_L-0)\leq+\infty$.

We first show that \be\label{b45}\kappa_L>0.\ee Suppose that
$\kappa_L=0$. By virtue of (4.6) in \cite{ACF5}, for any free
boundary point $X_0$ and $\e>0$, we have
\be\label{b47}\f1{r}\left|\fint_{\p B_r(X_0)}\psi_{Q_n}dS\right|\leq
C|\ld_{Q_n}|^{\f12},\ee if $B_r(X_0)\subset \O_L\cap\{\e<y<L-\e\}$
and $n$ is sufficiently large, where $C>0$ is a constant depending
only on $\e$. Taking $Q_n\ra +\infty$ in \eqref{b47} yields that
$$\f1{r}\left|\fint_{\p B_r(X_0)}\bar\psi_{\kappa_L}dS\right|\leq
C|\kappa_L|^{\f12}=0,$$ which together with
$\bar\psi_{\kappa_L}\equiv0$ in $\O_L\cap\{x<0,y>0\}$ implies that
$$\bar\psi_{\kappa_L}\equiv 0\ \ \text{in}\ \ \{(x,y)\mid 0<x<\e, \e<y<L-\e\}.$$
By using the unique continuation, we can conclude that
$\psi_{\kappa_L}\equiv0$ in $E_L$, which contradicts to the fact
$\bar\psi_{\kappa_L}=-1$ on $N_2$.

Finally, we will investigate the relation between $\kappa_L$ and
$h_{\kappa_L}$, where $h_{\kappa_L}\in(b,L]$ is the asymptotic
height of the free boundary $\Gamma_{\kappa_L}$. Consider the
following two cases.

{\bf Case 1.}  $h_{\kappa_L}< L$ and
$g_{\kappa_L}(h_{\kappa_L}-0)=+\infty$. (See Figure \ref{f12})

\begin{figure}[!h]
\includegraphics[width=100mm]{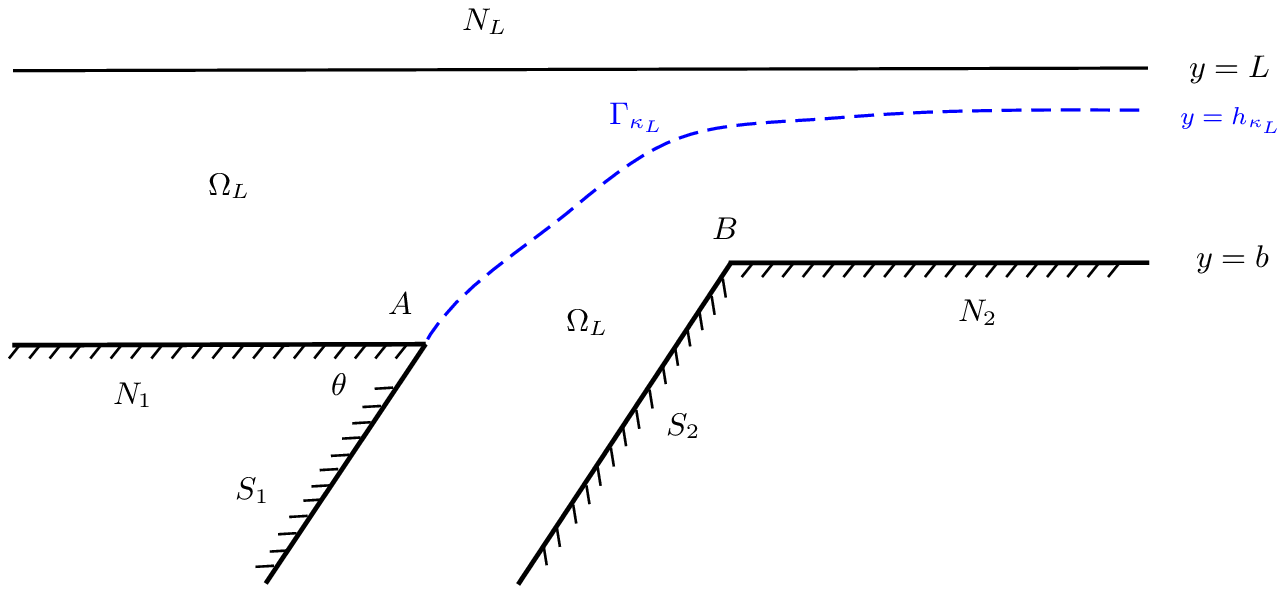}
\caption{Case 1}\label{f12}
\end{figure}

Similar to Step 2 in the proof of Theorem \ref{lb6} , we can obtain
that \be\label{cc3}\kappa_L=\f1{(h_{\kappa_L}-b)^2}.\ee

{\bf Case 2.} $h_{\kappa_L}= L$ and
$g_{\kappa_L}(L-0)\in(0,+\infty)$. (See Figure \ref{f13}).

\begin{figure}[!h]
\includegraphics[width=100mm]{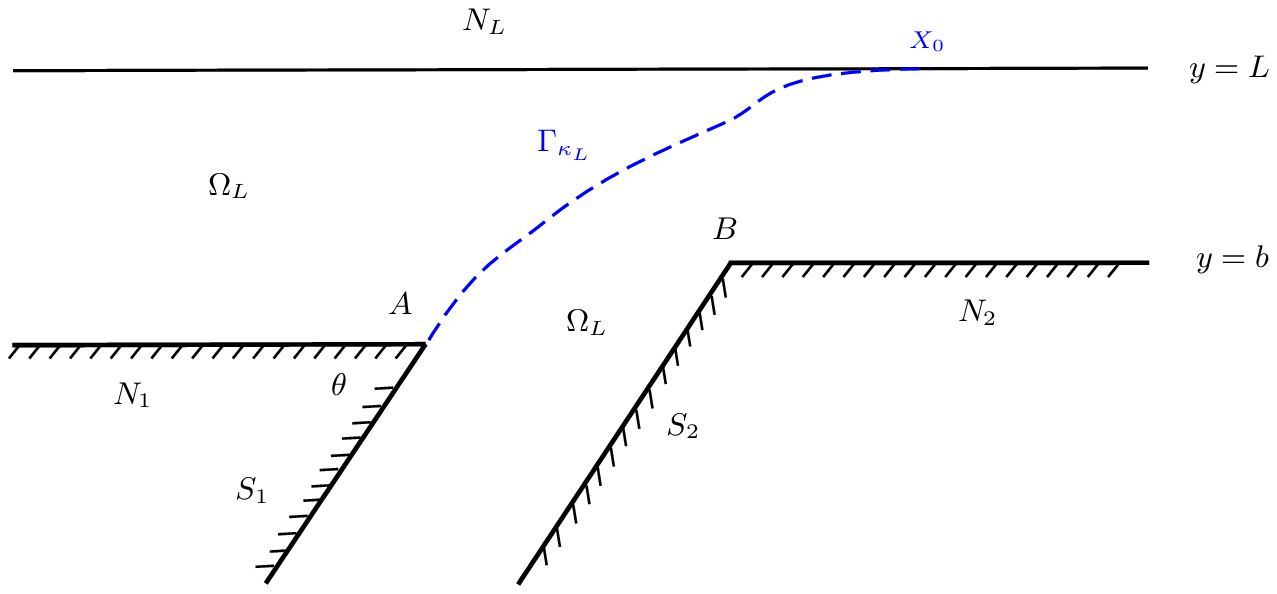}
\caption{Case 2}\label{f13}
\end{figure}

Denote $X_0=(g_{\kappa_L}(L-0),L)$. Similarly, we have that
$N_L\cup\Gamma_{\kappa_L}$ is $C^1$-smooth at $X_0$ and its tangent
is in the direction of positive $x$-axis. Moreover,
$\g\bar\psi_{\kappa_L}$ is uniformly continuous in a
$\{\bar\psi_{\kappa_L}<0\}$-neighborhood of $X_0$. Define
$\o(y)=\f{1}{L-b}\max\{y-b,0\}-1$, it is easy to check that
$$\bar\psi_{\kappa_L}(x,y)\geq\o(y)\ \ \text{in}\ \ E_L,$$ which implies that
$$\sqrt{\kappa_L}=\f{\p\bar\psi_{\kappa_L}}{\p\nu}\leq\f{\p\o}{\p\nu}=\f{1}{L-b}\ \ \text{at}\ \
X_0,$$ where $\nu=(0,1)$ is the outer normal vector. This implies
that \be\label{cc4}\kappa_L\leq\f{1}{(L-b)^2}.\ee

By using the similar arguments in the proof of Lemma \ref{lb8}, we
can obtain the uniqueness of $\kappa_L$ and $\bar\psi_{\kappa_L}$ to
the free boundary problem \eqref{b44}. Hence, one has
$$\f{\ld_L(Q)}{Q^2}\rightarrow\kappa_L,$$ and
$$\f{\psi_{\ld_L(Q),L}}{Q}\rightarrow\bar\psi_{\kappa_L}\ \text{uniformly in $E_L$},$$ as
$Q\ra+\infty$.

{\bf Step 3.} Finally, we will show that $\kappa_L$ is a uniform
constant for any large $L$, namely, there exists a $L_0$, such that
$\kappa_{L_1}=\kappa_{L_2}$ for any $L_2>L_1>L_0$. It follows from
\eqref{b43} that there exists a positive constant $C_2$ independent
of $L$, such that \be\label{c5}\kappa_L\leq C_2.\ee

By using the bounded gradient lemma 5.1 in Chapter 3 in \cite{FA1},
one has \be\label{cc5}|\g\bar\psi_{\kappa_L}|\leq C\sqrt{\kappa_L}\
\ \ \text{in}\ \ D\subset \O_L,\ee where
$D\cap\Gamma_{\kappa_L}\neq\varnothing$ and the constant $C$ depends
only on $D$. Denote $D=\O_L\cap B_{2a}(0)$, it is easy to check that
$D\cap\Gamma_{\kappa_L}\neq\varnothing$. Then there exist two points
$X_1\in \bar D\cap S_2$ and $X_2\in D\cap\Gamma_{\kappa_L}$, such
that $X_t=tX_1+(1-t)X_2\in D$ for any $t\in(0,1)$. It follows from
\eqref{cc5} that
$$1=\bar\psi_{\kappa_L}(X_2)-\bar\psi_{\kappa_L}(X_1)\leq |\g\bar\psi_{\kappa_L}(X_{t_0})||X_1-X_2|\leq C\sqrt{\kappa_L},$$ for $t_0\in(0,1)$,
where $C$ is a constant independent of $L$. This implies that there
exists a positive constant $C_1$ independent of $L$, such that
\be\label{c6}\kappa_L\geq C_1>0.\ee It follows from \eqref{cc3}-
\eqref{c6} that \be\label{cc6}h_{\kappa_L}\leq
b+\f{1}{\sqrt{\kappa_L}}\leq L_0,\ee where $L_0$ is a constant
independent of $L$.

Suppose that there exist two solutions
$(\bar\psi_{\kappa_{L_1}},\kappa_{L_1})$ and
$(\bar\psi_{\kappa_{L_2}},\kappa_{L_2})$ to the free boundary
problem \eqref{b44}, with $L_2>L_1>L_0$.

By virtue of \eqref{cc6}, we have
$$\text{the free boundary of $\bar\psi_{\kappa_{L_1}}$ lies below $\{y=L_1\}$},$$
and $$\text{the free boundary of $\bar\psi_{\kappa_{L_2}}$ lies
below $\{y=L_1\}$}.$$ Applying the similar arguments in the proof of
Lemma \ref{lb8}, we can obtain that
$\bar\psi_{\kappa_{L_1}}=\bar\psi_{\kappa_{L_2}}$ and
$\kappa_{L_1}=\kappa_{L_2}$.

\end{proof}

\begin{remark}\label{re3}
By virtue of Lemma \ref{lb10}, there exists a constant
$\kappa\in(0,+\infty)$, such that
$$\f{\ld_L(Q)}{Q^2}\rightarrow\kappa\ \ \text{and}\ \ \f{\psi_{\ld_L(Q),L}}{Q}\rightarrow\bar\psi_{\kappa}\ \text{uniformly in $E_L$},$$ as
$Q\ra+\infty$, for any $L>L_0$, where $E_L$ is defined as in
\eqref{b44}.
\end{remark}


Next, we will give the uniform estimate of the asymptotic height
$h_L$ of the free boundary.

\begin{lemma}\label{lc1}  For any $Q>0$, there exists a positive constant $C$ independent of $L$, such that
$$h_L\leq C,$$ where $h_L$ is the asymptotic height of the free boundary of
$\psi_{\ld_L(Q),L}$.\\

\end{lemma}

\begin{proof} Suppose not, we assume that there exists a sequence
$\{L_n\}$ with $L_n\ra+\infty$, such that $h_{L_n}\ra+\infty$. Note
that
\be\label{c1}\ld_{L_n}(Q)=\f{Q^2}{(h_{L_n}-b)^2}-\f{L_n^2}{(L_n-h_{L_n})^2}.\ee
Denote $(\psi_{\ld_{L_n}(Q),L_n},\ld_{L_n}(Q))$ as the corresponding
solution to the truncated injection flow problem 1 for any $Q>0$. By
virtue of \eqref{b43}, there exists a subsequence $\{L_n\}$, such
that
$$\ld_{L_n}(Q)\rightarrow\ld,$$ and
$$\psi_{\ld_{L_n}(Q),L_n}\rightarrow\psi_{\ld}\ \text{in
$H^1_{loc}(\O)$ and uniformly in any compact subset of $\O$},$$ as
$L_n\ra+\infty$. Moreover, $\psi_{\ld}(x,y)$ is monotone increasing
with respect to $y$ and decreasing with respect to $x$, which
implies that the free boundary of $\psi_{\ld}$ can be denoted as
$$\Gamma_{\ld}=\O\cap\{x>0\}\cap\{\psi_{\ld}=0\}: y=k_\ld(x)\ \ \text{for any $x>0$}.$$
Here,
$k_{\ld}(x)$ is continuous and strictly monotone increasing with
respect to $x$, $k_{\ld}(0)=0$ and $k_{\ld}(x)\ra +\infty$ as
$x\ra+\infty$. It follows from \eqref{c1} that
\be\label{c2}\ld\leq-1.\ee Furthermore, the free boundary
$\Gamma_\ld$ is continuous differentiable at $A$, namely,
$k_\ld'(0+0)=0$.

In view of the condition (7) in Definition \ref{def2}, one has
$$\psi_{\ld_{L_n}(Q),L_n}(x,y)\leq \max\{y,0\}\ \ \text{in}\ \
\O_{L_n}.$$ The strong maximum principle gives that
$$\psi_{\ld}(x,y)< y \ \text{in}\ \ \O\cap\{\psi_\ld>0\}.$$
Therefore, for small $r>0$, there exists a small $\delta>0$, such
that
$$\max\{y,0\}\geq(1+\delta)\psi_{\ld}\ \ \text{on}\ \
\p (B_r(0)\cap\{\psi_\ld>0\}).$$ It follows from the maximum
principle that \be\label{c3}\max\{y,0\}\geq(1+\delta)\psi_{\ld}\ \
\text{in}\ \ B_r(0)\cap\{\psi_\ld>0\}.\ee Define a blow-up sequence
$\t\psi_{n}(\t X)=\f{\psi_\ld(r_n\t X)}{r_n}$ with $r_n\ra0$, it
follows from \eqref{c3} that \be\label{c4}\max\{\t
y,0\}\geq(1+\delta)\t\psi_{n}(\t X)\ \ \text{in}\ \
B_1(0)\cap\{\psi_n>0\}.\ee Denote $\t\psi_0$ as the blow-up limit of
$\t\psi_n$, it follows from \eqref{c4} and the similar arguments in
the proof of Theorem \ref{lb6} that
$$\t\psi_0(\t X)=\max\{\sqrt{-\ld}\t y,0\}\ \ \text{and}\ \ \max\{\t y,0\}\geq(1+\delta)\t\psi_0(\t X)\ \ \text{in}\ \
B_1(0)\cap\{\psi_0>0\}.$$ This gives that
$$1\geq(1+\delta)\f{\p\t\psi_0}{\p\nu}=(1+\delta)\sqrt{-\ld}\ \ \text{at}\ \ 0,$$ where $\nu=(0,1)$ is the inner normal
vector. It leads a contradiction with \eqref{c2}.

\end{proof}

\section{The proof of The main results}

Based on the results in previous sections, we will complete the
proof of Theorem \ref{th1} - Theorem \ref{th3} in this section.

\begin{theorem}\label{ld1} For any $Q>0$, there exist a unique $\ld>-1$ and a unique solution $(\psi_{\ld},\Gamma_{\ld})$ to the injection flow
problem 1.

\end{theorem}
\begin{proof}{\bf Step 1.}  It follows from \eqref{b43} that there
exists a positive constant $C$ independent of $Q$ and $L$, such that
$$|\ld_L|\leq CQ^2\ \ \text{for any $L>b$ and $Q>0$}.$$
By virtue of Lemma \ref{lc1}, one has
$$h_L\leq C\ \ \text{and}\ \ \ \ld_L=\f{Q^2}{(h_L-b)^2}-\f{L^2}{(L-h_L)^2}=\f{Q^2}{(h_L-b)^2}-\f{1}{\left(1-\f{h_L}{L}\right)^2}.$$

Then there exist a sequence $\{L_n\}$, a constant $\ld$ and a $h
>b$, such that
$$\ld_{L_n}\rightarrow\ld,\ \ h_{L_n}\ra h,$$ and
$$\psi_{\ld_{L_n},L_n}\rightarrow\psi_{\ld}\ \text{in
$H^1_{loc}(\O)$ and uniformly in any compact subset of $\O$},$$ as
$L_n\ra+\infty$. Obviously, $\ld=\f{Q^2}{(h-b)^2}-1$. By using the
similar arguments in Lemma 6.2 in \cite{ACF4}, we can show that
$\psi_{\ld}$ is a local minimizer to the variational problem
$(P_{\ld})$, namely,
$$P_{\ld}:\ \ J_{D}(\psi_{\ld})=\min J_{D}(\psi)\ \ \text{for any $\psi\in K$ and
$\psi=\psi_{\ld}$ on $\p D$},$$ where
$$J_{D}(\psi)=\int_{D}\left|\nabla\psi-(\ld_1I_{\{\psi<0\}}+\ld_2I_{\{\psi>0\}}+\ld_0I_{\{\psi=0\}})I_{\{x>0\}}e\right|^2dxdy$$
for any bounded domain $D\subset \O$, where $\ld_1=\f{Q}{h-b}$ and
$\ld_2=1$.

{\bf Step 2.} Since $\psi_\ld$ is a local minimizer, we can conclude
that $\psi_\ld$ is a harmonic in $\O\setminus\Gamma$. Moreover, the
free boundary $\Gamma_\ld: y=k_\ld(x)$ of $\psi_{\ld}$ satisfies the
continuous fit condition $k_{\ld}(0)=0$ and the smooth fit condition
\eqref{a3}, where $k_{\ld}(x)$ is continuous and strictly monotone
increasing with respect to $x$, and $k_{\ld}(x)\ra h$ as
$x\ra+\infty$. It follows from the condition (7) in Definition
\ref{def2} that \be\label{d0}y-h\leq\psi_\ld^+(x,y)\leq y\ \
\text{in $\O\cap\{y>0\}$}.\ee

Hence, the conditions (1)-(5) and (7) in Definition \ref{def1} hold.

{\bf Step 3.} In this step, we will verify the condition (6) in
Definition \ref{def1}. Denote $\phi(x,y)=\psi_{\ld}(x,y)-y$ and
$\phi_n(x,y)=\phi(x-n,y)$, it follows from \eqref{d0} that
$$\Delta\phi_n=0\ \ \text{and}\ \ -h\leq\phi_n\leq0\ \ \text{in}\ \ \{x<n,y>0\}.$$
By using the elliptic estimate, there exists a subsequence
$\{\phi_n\}$, such that $$\phi_n\rightarrow \phi_0\ \ \text{in}\ \
\{-\infty<x<+\infty,0<y<+\infty\},$$ and $\phi_0$ satisfies
$$\Delta\phi_0=0\ \ \text{and}\ \ -h\leq\phi_0\leq0\ \ \text{in}\ \
\{-\infty<x<+\infty,0<y<+\infty\},\ \ \text{and}\ \ \phi_0(x,0)=0.$$
Then $\phi_0\equiv0$ in $\{-\infty<x<+\infty,0<y<+\infty\}$, and
thus \be\label{d1}\psi_\ld(x,y)\ra y\ \ \text{uniformly in any
compact subset of $(0,+\infty)$, as $x\ra-\infty$},\ee Along the
similar arguments in the proof of \eqref{b220}, one has
\be\label{d2}\left|\psi_{\ld}(x,y)-\f{Q(y\cos\th-x\sin\th)}{a\sin\th-b\cos\th}\right|\ra
0 \ \ \text{uniformly in any compact subset of $S$,}\ee as
$y\ra-\infty$, where $S=\{(x,y)\mid
y\cot\th<x<(y-b)\cot\th+a,-\infty<y<+\infty\}$.

Next, we consider the asymptotic behavior of $\psi_{\ld}$ in the
downstream. For any blow-up sequence
$\psi_n(x,y)=\psi_{\ld_L,L}(x+n,y)$ for $x>-\f n2$, such that
$$\psi_n(x,y)\rightarrow\psi_0(x,y)\  \ \text{uniformly in any compact subset of $(0,+\infty)$, as
$x\ra+\infty$},$$ and $\psi_0$ satisfies
$$\left\{\ba{ll} &\Delta\psi_0=0  \ \text{in}~~\mathbb{R}_b^2\setminus\{y=h\},\\
&\psi_0(x,b)=-Q\ \text{and}\ \psi_0(x,h)=0\ \ \text{for
$-\infty<x<+\infty$},\\
&0\leq\psi_0(x,y)\leq h\ \ \text{in}\ \
\{-\infty<x<+\infty\}\times\{0<y<h\}, \ea\right.
$$ where $\mathbb{R}^2_b=\{(x,y)\mid -\infty<x<+\infty,y>b\}$.

 Then one has that $\{y=h\}$ is the free boundary of $\psi_0$ and $\psi_0=\f{Q(y-b)}{h-b}-Q$ in $\{-\infty<x<+\infty\}\times\{0<y<h\}$.
In view of the condition (3) in Definition \ref{def1}, we can
conclude that $\f{\p\psi_0(x,h+0)}{\p y}=1$, and thus
$\psi_0(x,y)=y-h$ in $\{-\infty<x<+\infty\}\times\{y>h\}$.
Therefore, the boundary value problem above possesses a unique
solution
$$\psi_0(x,y)=\left\{\ba{ll} \f{Q(h-y)}{h-b},\
&\text{if $b<y<h$},\\
y-h,\ &\text{if $h<y<+\infty$}.\ea\right.$$

Finally, we will verify the convergence of $\g\psi_\ld$ in the far
field. For any sequence $X_n=(x_n,y_n)\in\O\cap\{\psi_\ld>0\}$ with
$\rho_n=|X_n|\ra+\infty$, we next consider the following two cases.

{\bf Case 1.} $y_n>\e |x_n|$ for $\e>0$, or $x_n<0$ and
$\f{y_n}{x_n}\ra 0$. Define $\t Y_n=\f{X_n}{\rho_n}$ and a blow-up
sequence
$$\t\psi_{\rho_n}(\t
X)=\f{\psi_{\ld}(\rho_n\t X)}{\rho_n}.$$ Then one has
$$\t Y_n\rightarrow\t
Y_0=(\t x_0,\t y_0)\ \text{and}\ \t\psi_{\rho_n}\ra \t\phi\ \
\text{uniformly in any compact subset of $\mathbb{R}^2\cap\{\t
y>0\}$}.$$ By virtue of \eqref{d0}, one has
$$\t y-\f{h}{\rho_n}\leq\t\psi_{\rho_n}(\t X)\leq \t y \ \ \text{in}\ \
\mathbb{R}^2\cap\{\psi_{\rho_n}>0\},$$ which implies that $\t\phi(\t
X)=\max\{\t y,0\}$.

If $y_n>\e |x_n|$ for $\e>0$, it is easy to check that $$|\t Y_0|=1\
\ \text{and} \ \ \t y_0\geq\f{\e}{\sqrt{1+\e^2}}.$$ By virtue of
elliptic regularity, one has
$$\t\psi_{\rho_n}\rightarrow\t\phi\ \ \text{in $C^{2,\alpha}(B_r(\t
Y))$},\ \ \alpha\in(0,1),$$ for $0<r<\f{\e}{4\sqrt{1+\e^2}}$. Thus
one has
$$\g\t\psi_{\rho_n}\left(\f{X_n}{\rho_n}\right)\ra\g\t\phi(\t Y_0)=(0,1).$$

If $x_n<0$ and $\f{y_n}{x_n}\ra 0$, one has $$\t Y_0=(-1,0).$$
Applying elliptic estimates, one has
$$\t\psi_{\rho_n}\rightarrow\t\phi\ \ \text{in $C^{2,\alpha}(D\cup T)$},\ \ \alpha\in(0,1),$$ where $D=B_{2r}(\t Y)\cap\{\t y>0\}$ and $T=\{(\t x,0)\mid |\t x+1|<r\}$ for $r>0$. Consequently,
$$\g\t\psi_{\rho_n}\left(\f{X_n}{\rho_n}\right)\ra\g\t\phi(\t Y_0)=(0,1).$$

{\bf Case 2.} $x_n>0$ and $\f{y_n}{x_n}\ra 0$ and
$y_n-k_\ld(x_n)\ra+\infty$ as $n\ra+\infty$.  Define a blow-up
sequence $\psi_{r_n}(\t X)=\f{\psi_\ld(Z_n+r_n\t X)}{r_n}$ and
$\psi_0$ is the blow-up limit of $\psi_{r_n}$, where
$r_n=y_n-k_\ld(x_n)$ and $Z_n=(x_n,k_\ld(x_n))$. The inequality
\eqref{d0} gives that
$$\t y+\f{k_\ld(x_n)-h}{r_n}\leq\psi^+_{r_n}(\t X)\leq \t
y+\f{k_\ld(x_n)}{r_n}\ \ \text{in}\ \ \{\t y>0\},$$ which implies
that
$$\psi_0(\t X)=\max\{ \t y,0\}\ \ \text{in}\ \ \{\t y>0\}.$$ Since
$-Q\leq\psi_\ld<0$ in $\O\cap\{\psi_\ld<0\}$, which implies that
$\psi^-_0=0$. Therefore, $\psi_0$ is 1-plane solution, and
$\psi_0(\t X)=\max\{ \t y,0\}$.
 The elliptic regularity gives that
$$\psi_{r_n}\rightarrow\psi_0\ \ \text{in $C^{2,\alpha}(B_{\f{1}{4}}(X_1))$},\ \ \alpha\in(0,1),\ \
X_1=(0,1).$$ Thus one has
$$\g\psi_\ld(X_n)=\g\psi_{r_n}(X_1)\rightarrow(0,1)\ \ \text{as}\ \
n\ra+\infty.$$ This gives that
$\g\psi_\ld(x,y)\rightarrow\g\psi(X_1)=(0,1)$ as
$x^2+y^2\rightarrow+\infty$ with
dist($(x,y),\Gamma$)$\rightarrow+\infty$ and $x>0$.

{\bf Step 2.} In this step, we will obtain the uniqueness of the
injection flow problem 1. Suppose that there exists another
different solution $(\t\psi_{\t\ld},\t\ld)$ to the injection flow
problem 1. In view of \eqref{a4}, one has
\be\label{d3}\psi_\ld(X)-y=o(|X|), \ \ \psi_\ld(X)>0, \text{as
$|X|\ra+\infty$,}\ee and \be\label{d4}\t\psi_{\t\ld}(X)-y=o(|X|), \
\ \t\psi_{\t\ld}(X)>0, \text{as $|X|\ra+\infty$.}\ee Without loss of
generality, we assume that $\ld\leq\t\ld$. It is easy to check that
\be\label{d5}\lim_{x\ra+\infty}k_{\ld}(x)=h=\f{Q}{\sqrt{1+\ld}}+b\geq\f{Q}{\sqrt{1+\t\ld}}+b=\t
h=\lim_{x\ra+\infty}\t k_{\t\ld}(x).\ee  Define
$\psi_{\ld,\e}(x,y)=\psi_\ld(x,y-\e)$ for any $\e\geq0$. Since the
asymptotic heights of the free boundaries $\Gamma_\ld$ and
$\t\Gamma_{\t\ld}$ are finite, it follows from \eqref{d5} that we
can take $\e_0\geq0$ to be the smallest one, such that
\be\label{d7}\text{the free boundary of $\psi_{\ld,\e_0}$ lies above
the free boundary of $\t\psi_{\t\ld}$}.\ee Denote
$\O^+=\O\cap\{\psi_{\ld,\e_0}>0\}\cap\{\t\psi_{\t\ld}>0\}$ and
$\o(X)=\t\psi_{\t\ld}-\psi_{\ld,\e_0}$, it follows from \eqref{d3},
\eqref{d4} and \eqref{d7} that
$$\o(X)\geq 0\ \ \text{on}\ \ \p\O^+\ \ \text{and}\ \ \lim_{r\ra +\infty}\f{m(r)}{r}\ra
0,$$ where $r=|X|$ and $m(r)=\min_{|X|=r}\o(X)$. Applying the
Phragm$\grave{\text{e}}$n-Lindel\"of theorem in \cite{G1}, one has
$$\o(X)\geq 0\ \ \text{in}\ \ \O^+,$$ which implies that
\be\label{d8}\t\psi_{\t\ld}\geq\psi_{\ld,\e_0}\ \ \text{in
$\O\cap\{\psi_{\ld,\e_0}>0\}$}.\ee

By virtue of the asymptotic behavior of $\psi_{\ld,\e_0}$ and
$\t\psi_{\t\ld}$, it follows from the similar arguments in the step
4 in the proof of Theorem \ref{lb6} that
\be\label{d9}\t\psi_{\t\ld}\geq\psi_{\ld,\e_0}\ \ \text{in
$\O\cap\{\t\psi_{\t\ld}<0\}$}.\ee

Next, we consider two cases in the following.

{\bf Case 1.} $\e_0>0$. In view of \eqref{d7}, we can take a free
boundary point $X_0$ with $|X_0|<+\infty$. Applying the strong
maximum principle, one has
$$\t\psi_{\t\ld}>\psi_{\ld,\e_0}\ \ \text{in
$\O\cap\{\psi_{\ld,\e_0}>0\}$ and}\ \
\t\psi_{\t\ld}>\psi_{\ld,\e_0}\ \ \text{in
$\O\cap\{\t\psi_{\t\ld}<0\}$}.$$ Since the free boundary
$\t\Gamma_{\t\ld}$ and $\Gamma^{\e_0}_{\ld}$ are analytic at $X_0$,
it follows from Hopf's lemma that
$$|\g\psi_{\ld,\e_0}^-|=-\f{\p\psi^-_{\ld,\e_0}}{\p\nu}>-\f{\p\t\psi_{\t\ld}^-}{\p\nu}=|\g\t\psi_{\t\ld}^-|\ \ \text{and}\ \
|\g\psi_{\ld,\e_0}^+|=\f{\p\psi^+_{\ld,\e_0}}{\p\nu}<\f{\p\t\psi_{\t\ld}^+}{\p\nu}=|\g\t\psi_{\t\ld}^+|\
\ \text{at}\ \ X_0,$$ where $\nu$ is the inner normal vector to
$\p\{\t\psi_{\t\ld}>0\}$ at $X_0$. Those give that
$$\ld=|\g\psi_{\ld,\e_0}^-|^2-|\g\psi_{\ld,\e_0}^+|^2>|\g\t\psi_{\t\ld}^-|^2-|\g\t\psi_{\t\ld}^+|^2=\t\ld \ \ \text{at}\ \ X_0,$$
which contradicts to our assumption $\ld\leq\t\ld$.

{\bf Case 2.} $\e_0=0$. Similar to \eqref{b30}, we can show that
$\ld\cdot\t\ld>0$. Without loss of generality, one assume that
$0>\t\ld\geq\ld$. Therefore, we can obtain a contradiction by using
the similar arguments in Subcase 2.1 in the proof of Lemma
\ref{lb7}.

\end{proof}

By virtue of Theorem \ref{ld1}, we complete the proof of Theorem
\ref{th1}.

Due to the uniqueness of $\ld$ for any $Q>0$, we can define a
function $\ld=\ld(Q)$ for any $Q>0$. We next consider the relation
between $\ld(Q)$ and $Q>0$, and complete the proof of Theorem
\ref{th2}.

\begin{proof}[Proof of Theorem \ref{th2}]
(1). For any $Q_1>Q_2>0$, there exist two solutions
$(\psi_{\ld(Q_1)},\ld(Q_1))$ and $(\psi_{\ld(Q_2)},\ld(Q_2))$ to the
injection flow problem 1. We next show that
$$\ld(Q_1)>\ld(Q_2)\ \ \text{for any $Q_1>Q_2>0$}.$$
If not, then there exist $Q_1>Q_2>0$, such that
 $\ld(Q_1)\leq\ld(Q_2)$, and we consider the following two
 cases.

 {\bf Case 1.} $\ld(Q_1)=\ld(Q_2)=0$. Since $\ld(Q_1)$ and $\psi_{\ld(Q_1)}$ are unique for any given $Q_1>0$,
there exists a sequence $\ld_{L_n}(Q_1)$ with $\ld_{L_n}(Q_1)=0$,
such that
$$\psi_{\ld_{L_n}(Q_1),L_n}\rightarrow\psi_{\ld(Q_1)}\ \text{in
$H^1_{loc}(\O)$ and uniformly in any compact subset of $\O$},$$ as
$L_n\ra+\infty$. By virtue of \eqref{b34}, one has
$$(a\sin\th-b\cos\th)^{\f\pi\th}Q_1^{1-\f\pi\th}=L_n-(L_n-b)^{\f\pi\th}(L_n+Q_1)^{1-\f\pi\th}.$$
Set $t_n=L_n+Q_1$, one has
$$(a\sin\th-b\cos\th)^{\f\pi\th}Q_1^{1-\f\pi\th}=t_n\left(1-\left(1-\f{Q_1+b}{t_n}\right)^{\f{\pi}{\th}}\right)-Q_1
\rightarrow\left(\f{\pi}{\th}-1\right)Q_1+\f{b\pi}{\th},$$ as
$t_n\rightarrow+\infty$. Then one has
 \be\label{d10}(a\sin\th-b\cos\th)^{\f\pi\th}=\left(\f\pi\th-1\right)Q_1^{\f\pi\th}+\f{b\pi}{\th}Q_1^{\f\pi\th-1}.\ee
 Similarly, we have
$$(a\sin\th-b\cos\th)^{\f\pi\th}=\left(\f\pi\th-1\right)Q_2^{\f\pi\th}+\f{b\pi}{\th}Q_2^{\f\pi\th-1},$$
 which together with \eqref{d10} implies that $Q_1=Q_2$. This leads a
contradiction.

  {\bf Case 2.} $\ld(Q_1)\neq0$ or $\ld(Q_2)\neq0$.

Since $Q_1>Q_2$, one has
\be\label{d11}h_1=\f{Q_1}{\sqrt{\ld(Q_1)+1}}+b>\f{Q_2}{\sqrt{\ld(Q_2)+1}}+b=h_2\
\ \text{and}\ \ k_{\ld(Q_1)}(x)>k_{\ld(Q_2)}(x)\ee for sufficiently
large $x>0$.

Define a function $\psi_{\ld(Q_1),\e}(x,y)=\psi_{\ld(Q_1)}(x,y-\e)$
for $\e\geq0$. In view of \eqref{d11}, let $\e_0\geq0$ be the
smallest one, such that $$\text{the free boundary of
$\psi_{\ld(Q_1),\e_0}$ lies above the free boundary of
$\psi_{\ld(Q_2)}$}.$$ Similar to the proof of Lemma \ref{ld1}, by
using the Phragm$\grave{\text{e}}$n-Lindel\"of theorem in \cite{G1}
and the asymptotic behavior of $\psi_{\ld(Q_1),\e_0}$ and
$\psi_{\ld(Q_2)}$, we have $$\psi_{\ld(Q_1),\e_0}\leq\psi_{\ld(Q_2)}
\ \text{in $\O\cap\{\psi_{\ld(Q_1),\e_0}>0\}$ and }
\psi_{\ld(Q_1),\e_0}\leq\psi_{\ld(Q_2)}\ \ \text{in
$\O\cap\{\psi_{\ld(Q_2)}<0\}$}.$$ Then we can obtain a contradiction
by using the similar arguments in the proof of Lemma \ref{ld1}.

By virtue of the uniqueness of the solution $(\psi_\ld,\ld)$, it
follows from the similar arguments in the proof of Lemma \ref{lb9}
that $\ld(Q)$ is continuous for any $Q>0$.

(2). Next, we will show that there exists a
$\underline\ld\in(-1,0)$, such that $\ld(Q)\ra\underline\ld$ as
$Q\rightarrow0$. The monotonicity of $\ld(Q)$ gives that there
exists $\underline\ld\geq-1$, such that
$\ld(Q)\rightarrow\underline\ld$ as $Q\rightarrow0$. It suffices to
exclude the case $\ld=-1$. For any sequence $\{Q_n\}$ with $Q_n>0$
and $Q_n\rightarrow 0$, such that
$$\ld(Q_n)\rightarrow\underline\ld,\ \psi_{\ld(Q_n)}\rightarrow\psi_{\underline\ld}\ \text{in
$H^1_{loc}(\O)$ and uniformly in any compact subset of $\O$},$$ as
$Q_n\ra0$. Moreover, $\p_x\psi_{\underline\ld}\leq 0$ and
$\p_y\psi_{\underline\ld}\geq0$ in
$\O\cap\{\psi_{\underline\ld}>0\}$, which implies that
$$\O\cap\{x>0\}\cap\{\psi_{\underline\ld}>0\}=\O\cap\{x>0\}\cap\{y>k_{\underline\ld}(x)\},$$
where $k_{\underline\ld}(x)$ is monotone increasing for $x>0$, and
$k_{\underline\ld}(0)=0$.

Suppose that $\underline\ld=-1$. Define $\o(y)=\max\{y,0\}$, it
follows from the condition (7) in Definition \ref{def1} that
$\o(y)\geq\psi_{\underline\ld}(x,y)$ in $\O$. The strong maximum
principle gives that $$\psi_{\underline\ld}<y\ \ \text{in}\ \
\O\cap\{\psi_{\underline\ld}>0\}.$$ For any small $r>0$, it follows
from the proof of \eqref{b37} that there exists a small $\delta>0$,
such that
$$\o\geq(1+\delta)\psi_{\underline\ld}\ \ \text{in}\ \
B_r(A)\cap\{\psi_{\underline\ld}>0\},$$ which gives that
\be\label{d13}1\geq
(1+\delta)\f{\psi_{\underline\ld}}{\p\nu}=(1+\delta)\sqrt{-\underline\ld}=1+\delta\
\ \text{at}\ \ A,\ee where $\nu=(0,1)$ is inner normal vector. This
leads a contradiction.

Similar to the proof of \eqref{b38}, one has
\be\label{d12}k_{\underline\ld}(x)\equiv b\ \ \text{for any
$x\in(a,+\infty)$}.\ee In fact, if there exists an
$x_0\in(a,+\infty)$, such that $k_{\underline\ld}(x_0)=b$ and
$k_{\underline\ld}(x)>b$ for any $x\in(x_0,+\infty)$. The asymptotic
behavior of $\psi_{\underline\ld}$ gives that
$$\underline\ld=-1,$$ which contradicts to $\underline\ld>-1$.
Similar to the proof of \eqref{b40}, we can show that
$\bar\Gamma_{\underline\ld}\cap N_2=\varnothing$.

(3). In this step, we will show that
$$\f{\ld(Q)}{Q^2}\rightarrow\kappa\in(0,+\infty)\ \ \text{as}\ \
Q\rightarrow+\infty.$$

Set $\psi_Q=\f{\psi_{\ld(Q)}}{Q}$, $\ld_Q=\f{\ld(Q)}{Q^2}$ and $h_Q$
is the asymptotic height of the free boundary of $\psi_Q$. By virtue
of Lemma \ref{lb10}, one has
\be\label{d15}\ld_Q=\f1{(h_Q-b)^2}-\f{1}{Q^2}\leq C,\ \ \text{where
$C>0$ is a constant independent of $Q$}.\ee It follows from the
proof of \eqref{b45} that there exists a $c>0$ independent of $Q$,
such that \be\label{d16}\ld_Q\geq c>0.\ee In view of \eqref{d15} and
\eqref{d16}, there exist two positive constants $C_1$ and $C_2$
independent of $Q$, such that\be\label{d17}C_1\leq h_Q-b\leq C_2.\ee

Therefore, for any sequence $\{Q_n\}$ with $Q_n\ra+\infty$, such
that
$$\ld_{Q_n}\rightarrow \kappa, \ \ h_{Q_n}\ra h_\kappa\ \
\text{and}\ \ \psi_{\ld_{Q_n}}\rightarrow\bar\psi_{\kappa}\
\text{uniformly in $\O$, as $Q_n\ra+\infty$}.$$  It is easy to check
that $\bar\psi_\kappa=0$ in $\O\cap\{(x,y)\mid x\leq 0, y\geq 0\}$.
Similar to Lemma \ref{lb10}, $\bar\psi_{\kappa}$ is a solution of
the following free boundary problem
\be\label{d18}\left\{\ba{ll} &\Delta\bar\psi_{\kappa}=0  \ \text{in}~~E\cap\{\bar\psi_{\kappa}<0\},\\
&\bar\psi_{\kappa}=0,\ \ \left|\f{\p\bar\psi_{\kappa}}{\p\nu}\right|^2=\kappa\ \ \text{on $\Gamma_{\kappa}$},\\
&\bar\psi_{\kappa}=0\ \text{on}\ S_1,\ \bar\psi_{\kappa}=-1 \
\text{on} \ N_2\cup S_2, \ea\right. \ee where
$E=\O\setminus(\{(x,y)\mid x\leq 0,y\geq 0\})$ and
$\Gamma_{\kappa}=E\cap\{\bar\psi_\kappa<0\}$ is the free boundary of
$\bar\psi_{\kappa}$.

By using the similar arguments in the proof of Lemma \ref{lb10}, we
can obtain the uniqueness of $(\bar\psi_{\kappa},\kappa)$ to the
free boundary problem \eqref{d18}.

\end{proof}
Based on the proof of Theorem \ref{th2}, we can obtain the existence
and uniqueness of the solution to the injection flow problem 2.
\begin{corollary}\label{ld3} For any $\ld\in(\underline\ld,+\infty)$, there exist a unique $Q=Q(\ld)>0$ and a unique solution $(\psi_{Q},\Gamma_{Q})$ to the injection flow
problem 2. Furthermore,\\
(1) $Q(\underline\ld+0)=\lim_{\ld\ra\underline\ld^+}Q(\ld)=0$ and
$\f{Q^2(\ld)}{\ld}\ra \f1\kappa$ as $\ld\ra+\infty$.\\
(2) $Q(0)>0$ is uniquely determined by
$$(a\sin\th-b\cos\th)^{\f\pi\th}=\left(\f\pi\th-1\right)Q^{\f\pi\th}-\f{b\pi}{\th}Q^{\f\pi\th-1}.$$

\end{corollary}

Hence, Theorem \ref{th3} follows from Corollary \ref{ld3}
immediately.

\

 {\bf Conflict of interest. } The authors declare that they have no
conflict of interest.

\

{\bf Acknowledgments.} The authors would like to thank the referees
for their helpful suggestions and careful reading which improve this
paper. This work was done while the first author was visiting The
Institute of Mathematical Science at The Chinese University of Hong
Kong, and he would like to thank Professor Zhouping Xin for warm
hospitality and many helpful discussions.

\

\bibliographystyle{plain}

\begin{thebibliography}{10}

\bibitem{A1}
S. Agmon, A. Douglis, L. Nirenberg, \newblock Estimates near the
boundary for solutions of elliptic partial differential equations
satisfying general boundary conditions. I, {\it Comm. Pure Appl.
Math.,} {\bf 12}, 623-727, (1959).

\bibitem{AC1}
H. W. Alt, L. A. Caffarelli, \newblock Existence and regularity for
a minimum problem with free boundary, {\it J. Reine Angew. Math.,}
{\bf 325}, 105-144, (1981).



\bibitem{ACF3}
H. W. Alt, L. A. Caffarelli, A. Friedman, \newblock Axially
symmetric jet flows, {\it Arch. Rational Mech. Anal.,} {\bf 81},
97-149, (1983).

\bibitem{ACF4}
H. W. Alt, L. A. Caffarelli, A. Friedman, \newblock Variational
problems with two phases and their free boundaries, {\it Trans.
Amer. Math. Soc.,} {\bf 282}, 431-461, (1984).

\bibitem{ACF5}
H. W. Alt, L. A. Caffarelli, A. Friedman, \newblock Jets with two
fluids. I. One free boundary, {\it Indiana Univ. Math. J.,} {\bf
33}, 213-247, (1984).

\bibitem{ACF6}
H. W. Alt, L. A. Caffarelli, A. Friedman, \newblock Jets with two
fluids. II. Two free boundaries, {\it Indiana Univ. Math. J.,} {\bf
33}, 367-391, (1984).








\bibitem{C1}
L. A. Caffarelli, \newblock A Harnack inequality approach to the
regularity of free boundaries. Part I: Lipschitz free boundaries are
$C^{1,\alpha}$, {\it Rev. Mat. Iberoam.}, {\bf 3}, 139-162, (1987).
























\bibitem{DD}
L. L. Du, B. Duan, \newblock Global subsonic Euler flows in an
infinitely long axisymmetric nozzle, {\it J. Differential
Equations}, {\bf 250}, 813-847, (2011).


\bibitem{DXX}
L. L. Du, C. J. Xie, Z. P. Xin, \newblock Steady subsonic ideal
flows through an infinitely long nozzle with large vorticity, {\it
Comm. Math. Phys.}, {\bf 328}, 327-354, (2014).

\bibitem{DXY}
L. L. Du, Z. P. Xin, W. Yan,
\newblock Subsonic flows in a multi-dimensional
nozzle, {\em Arch. Rational Mech. Anal.}, {\bf 201}, 965--1012,
(2011).







\bibitem{CHO}
G. Chochua, M. Shyy, S. Thakur, et al,  A computational and
experimental investigation of turbulent jet and crossflow
interaction, {\it Numerical Heat Transfer: Part A: Applications},
{\bf 38}(6), 557--572, (2000).


\bibitem{FA1}
A. Friedman, {\it Variational principles and free-boundary
problems}, Pure and Applied Mathematics, John Wiley Sons, Inc., New
York, 1982.




\bibitem{FA3}
A. Friedman, \newblock Injection of ideal fluid from a slot into a
free stream, {\it  Arch. Ration. Mech. Anal.,} {\bf 94}, 335--361,
(1986).

\bibitem{FA4}
A. Friedman, \newblock Mathematics in industrial problems, The IMA
Volumes in Mathematics and its Applications, 16, Springer-verlag,
New York, 1988.




\bibitem{G1}
D. Gilbarg,
\newblock The Phragm$\grave{\text{e}}$n-Lindel\"of theorem for elliptic partical differential equations, {\it J. Rational Mech. Anal.,} {\bf 1}, 411--417, (1952).

\bibitem{GT}
D. Gilbarg, N. S. Trudinger,
\newblock Elliptic Partial Differential Equations of Second Order,
\newblock {\em Classics in Mathematics}. Springer-Verlag, Berlin, 2001.


\bibitem{MO} R. J. Margason, \newblock Fifty years of jet in cross flow research, In AGARD, Computational and Experimental
Assessment of Jets in Cross Flow 41 p. 1993.

\bibitem{MI}
 L. M. Milne-Thomson,
\newblock Theoretical Hydrodynamics, Macmillan, New York, 1968.



\bibitem{MB}
C. B. Morrey, Jr,
\newblock Multiple integrals in the calculus of variations, Springer-Verlag, Berlin, 1966.



\bibitem{MM}
S. Muppidi, K. Mahesh,  Direct numerical simulation of round
turbulent jets in crossflow, {\it  J. Fluid Mech.}, {\bf 574},
59--84, (2007).


\bibitem{SM}
L. K. Su, M. G. Mungal, Simultaneous measurements of scalar and
velocity field evolution in turbulent crossflowing jets, {\it  J.
Fluid Mech.}, {\bf 513}, 1--45, (2004).




















\bibitem{TL}
L. Ting, P. A. Libby, C. Ruger, \newblock The potential flow due to
a jet and a stream with different total pressures, {\it Polytechic
institute of Brooklyn, PIBAL Rept}, 1964.

\bibitem{XX1}
C. J. Xie, Z. P. Xin,
\newblock Global subsonic and subsonic-sonic flows through infinitely long nozzles,
\newblock {\em Indiana Univ.
Math. J.}, {\bf 56} (2007), 2991--3023.

\bibitem{XX2}
C. J. Xie, Z. P. Xin,
\newblock Global subsonic and subsonic-sonic flows through infinitely long axially symmetric nozzles,
\newblock {\em J. Differential Equations}, {\bf 248} (2010), 2657--2683.

\bibitem{XX3}
C. J. Xie, Z. P. Xin, Existence of global steady subsonic Euler
flows through infinitely long nozzle, {\it SIAM J. Math. Anal.},
{\bf 42} (2), 751--784, (2010).






\end{thebibliography}

\end{document}